\documentclass[a4paper,11pt]{amsart}
\usepackage[utf8]{inputenc}
\usepackage[english]{babel}
\usepackage{amsmath}
\usepackage{amsthm,amsfonts}
\usepackage{hyperref}
\usepackage{xcolor}
\usepackage{mathtools}
\usepackage{tikz-cd}
\usepackage{tkz-graph}
\usetikzlibrary{arrows.meta}
\usetikzlibrary{fit}
\usepackage{caption}

\usepackage[margin=3.5cm]{geometry}

\theoremstyle{plain}
\newtheorem{theorem}{Theorem}[section]
\newtheorem*{theorem_intro}{Theorem}
\newtheorem{proposition}[theorem]{Proposition}
\newtheorem{lemma}[theorem]{Lemma}
\newtheorem{corollary}[theorem]{Corollary}
\theoremstyle{definition}
\newtheorem{definition}[theorem]{Definition}
\newtheorem{remark}[theorem]{Remark}
\newtheorem{example}[theorem]{Example}

\newtheorem{notation}[theorem]{Notation}
\newtheorem{question}[theorem]{Question}
\newtheorem*{question_intro}{Question}

\newcommand{\op}{\mathrm{op}}
\newcommand{\len}{\mathrm{len}}
\newcommand{\Inj}{\mathrm{Inj}}
\newcommand{\dFl}{\mathrm{dFl}}

\newcommand{\bC}{\mathbf{C}}
\newcommand{\Cat}{\mathbf{Cat}}
\newcommand{\Set}{\mathbf{Set}}
\newcommand{\Reach}{\mathbf{Reach}}
\newcommand{\Quiver}{\mathbf{Quiver}}
\newcommand{\Pre}{\mathrm{Pre}}
\newcommand{\MC}{\mathrm{MC}}
\newcommand{\MH}{\mathrm{MH}}
\newcommand{\EMC}{\mathrm{EMC}}
\newcommand{\EMH}{\mathrm{EMH}}
\newcommand{\DMC}{\mathrm{DMC}}
\newcommand{\DMH}{\mathrm{DMH}}
\newcommand{\gir}{\mathrm{gir}}

\newcommand{\cA}{\mathcal{A}}

\usepackage{todonotes}

\title[EMH: diagonality and computations]{Eulerian magnitude homology: diagonality, injective words, and regular path homology}
\author{Luigi Caputi}
\address[LC]{Department of Mathematics, University of Bologna, Via Zamboni 33, 40126 Bologna, Italy}
\email{luigi.caputi@unibo.it}

\author{Giuliamaria Menara}
\address[GM]{Department of Informatics, Systems and Communication, University of Milano-Bicocca, Viale Sarca 336, 20145 Milano, Italy}
\email{giuliamaria.menara@unimib.it}


\begin{document}

\begin{abstract}
In this paper we explore the algebraic structure and combinatorial properties of eulerian magnitude homology. 
First, we analyze the diagonality conditions of eulerian magnitude homology, providing a characterization of complete graphs. Then, we 
construct the regular magnitude-path spectral sequence as the spectral sequence of the (filtered) injective nerve of the reachability category, and explore its consequences. Among others, we show that such spectral sequence converges to the complex of injective words on a digraph,  and yields characterization results for the regular path homology of diagonal directed graphs.
\end{abstract}
\maketitle
\section*{Introduction}

Magnitude homology, introduced by Hepworth and Willerton in \cite{hepworth2015categorifying}, is a combinatorially defined graph invariant that categorifies magnitude~\cite{leinster2019magnitude}. It provides a powerful tool for studying the combinatorial and categorical structure of graphs, see, \textit{eg}.~\cite{arXiv:1809.07240,zbMATH07283643,zbMATH07381918, zbMATH07370433,zbMATH07731261,arXiv:2312.01378, zbMATH07681908,asao2024girth,zbMATH07953420, zbMATH07868225,arXiv:2404.06689,arXiv:2411.02838,arXiv:2405.04748,zbMATH07860680} for a  (non-exhaustive) list of recent works.
Eulerian magnitude homology, introduced by Giusti and Menara in \cite{arXiv:2403.09248} (but see also~\cite{arXiv:2409.03472, arXiv:2410.10376}), is a variant of magnitude homology. Its definition refines the notion of magnitude chains by considering only chains with \textit{distinct} vertices in a graph. Although magnitude homology and eulerian magnitude homology are quite close homology theories -- they are related by a long exact sequence and, further, they coincide on posets -- they satisfy different algebraic and combinatorial properties. In this work, we aim to shed some light on  such differences, namely on the combinatorial and topological aspect. 

From a combinatorial viewpoint, a key result of this paper is the characterization of complete graphs via eulerian magnitude homology. In analogy with the notion of diagonality for magnitude homology, we say that a graph $G$ is \textit{regularly diagonal} if its associated eulerian magnitude homology~$\EMH_{*,*}(G)$ is diagonal. 

\begin{theorem_intro}[\textit{cf}.~Theorem~\ref{thm:charregdiag}]
    Let $G$ be a finite connected regularly diagonal undirected graph, with eulerian magnitude homology groups $\EMH_{k,k}(G)=0$  in all bidegrees $(k,k)$ for $k>n$. 
    Then, $G$ is the complete graph on $n+1$ vertices. 
\end{theorem_intro}

As all complete undirected graphs are regularly diagonal, this shows that eulerian magnitude homology detects the complete graphs. We recall here that this is not the case for magnitude homology, for which the family of undirected diagonal graphs is quite larger. When considering directed graphs, the story is different, and we do not know of any characterization of regularly diagonal digraphs. On the other hand, in analogy with what happens for magnitude homology of graphs, cones and, more generally, joins of regularly diagonal digraphs are regularly diagonal (see~Theorem~\ref{thm:joins}). We shall show in Section~\ref{sec:regularly diag digraphs} that directed trees, bipartite graphs with alternating orientations, and transitive tournaments, are other examples of regularly diagonal digraphs.

Besides the combinatorial appeal, our interest in (regular) diagonality of graphs stems from the fact that it is  useful in computations of \textit{regular path homology}~\cite{zbMATH06545866,arXiv:2208.14063}, a variant of the more classical and established path homology~\cite{zbMATH07253738}. With an eye towards these computations, we develop the analogue of the magnitude-path spectral sequence~\cite{zbMATH07731261} -- which we call the \textit{regular magnitude-path spectral sequence}. Using the same notation used in~\cite{zbMATH07731261}, we have:

\begin{theorem_intro}[\textit{cf}.~Theorem~\ref{thm:rmpss}]
Let $G$ be a finite directed graph.
    Then, the regular magnitude-path spectral sequence $IE_*^{*,*}$ satisfies the following properties:
    \begin{enumerate}
        \item The first page $IE_*^{*,1}$ is the eulerian magnitude homology of $G$.
        \item The diagonal of the second page $IE_*^{*,2}$ is the regular path homology.
        \item It converges to the homology of the complex of injective words on $G$.
        \item There is a  map of spectral sequences from $IE_*^{*,r}$ to the magnitude-path spectral sequence.
    \end{enumerate}
\end{theorem_intro}

The regular magnitude-path spectral sequence is the spectral sequence associated to the filtered simplicial set~$N^\iota(\Reach_G)$, which is the injective nerve of the reachability category of a directed graph $G$~\cite{zbMATH07844814}. The injective nerve is defined as the simplicial set on functors to $\Reach_G$ that are injective on objects -- see Section~\ref{sec:injnerve} -- and it agrees with the directed flag complex of the transitive closure $\Pre(G)$ (also known as \textit{reachability preorder}~\cite{zbMATH07868225,arXiv:2312.01378}) of~$G$. In particular, $N^\iota(\Reach_G)$ models the complex of injective words on $G$ -- see Definition~\ref{def:injwordsgraph} -- which in turn generalizes the classical complex of injective words~\cite{zbMATH03647758,zbMATH03811580,zbMATH07177613}. 
This leads to an interpretation of eulerian magnitude homology in terms of directed flag complexes and reachability preorders, and, in turn, to a description of the directed flag complex in terms of injective nerves, which we believe to be of independent interest.
As applications, we show the following:

\begin{theorem_intro}[\textit{cf}.~Proposition~\ref{prop:gjaja}]
    Let $G$ be a finite and connected regularly diagonal directed graph.  
    Then, the regular path homology of $G$ is isomorphic to the homology of the complex of injective words on $G$. 
\end{theorem_intro}

As a consequence, the reduced regular path homology of directed trees and transitive tournaments is trivial, whereas the regular path homology of complete graphs is the homology of a wedge of spheres. This last computation allows us to prove that regular path homology (and more generally, the pages of the regular magnitude-path spectral sequence) are not homotopy invariants -- see Corollary~\ref{cor:nonhinvpathhom}.

The results presented here open several exciting avenues for further research, and building upon these findings, we propose two questions, related to the above discussion, that we find particularly interesting and worth exploring:

\begin{question_intro}
    Is there a complete characterization of regularly diagonal directed graphs?
\end{question_intro}

\begin{question_intro}
    For a finite category $\bC$, when is the inclusion of the injective nerve
    \[
N^\iota(\bC)\to N(\bC)
\]
in the nerve of $\bC$ a weak equivalence of simplicial sets? 
\end{question_intro}

The characterization of complete undirected graphs  as by Theorem~\ref{thm:rmpss} was more recently shown also in~\cite{martin2025torsion} by Sazdanovic and Martin. 
The similarities with their work 
arose from independent work.

\subsection*{Outline}
The paper is organized as follows.
We recall in Section \ref{sec:EMH and RPH} some general background about first standard and eulerian magnitude homology, and further standard and regular path homology.
We proceed by investigating in Section \ref{sec:regularly diag graphs} diagonality conditions for eulerian magnitude homology of undirected graphs, and we then expand this study in Section \ref{sec:regularly diag digraphs} examining diagonality properties of eulerian magnitude homology of digraphs.
In Section \ref{sec:complex inj word digraph} we extend to the setting directed graphs the notion of complex of injective words, which will emerge as the target object of the regular magnitude-path spectral sequence developed in Section \ref{sec:regular MPSS}.

\subsection*{Conventions}

Unless otherwise specified, all (directed) graphs and categories are finite. Undirected graphs are simple. In the case of directed graphs, we shall not allow multiple directed edges, which means that the set of edges of a directed graph $G$ with vertices $V(G)$ is a subset of the ordered pairs $V(G)\times V(G)$.
All modules are taken over a base commutative ring~$R$. We shall use the following notations:

\begin{tabular}{ll}
$\rho(G)$ & the directed graph associated to an undirected graph $G$; \textit{cf}.~Notation~\ref{notation:digraphs}\\
$K_n$ & undirected complete graph on $n$ vertices\\
$C_n$ & undirected cyclic graph on $n$ vertices\\
$I_n$ & undirected linear graph on $n$ vertices\\
$L_n$ &
directed linear graph on $n$ vertices\\
$T_n$ & transitive tournament on $n+1$ vertices \\
$\gamma_n^s$ & girth of the maximal sub-cycle achievable by removing $s$ edges from $K_n$;\\
&\textit{cf}.~Notation~\ref{notation:maximal subcycle}.
\end{tabular}

\subsection*{Acknowledgments}

LC~was supported by the Starting Grant 101077154 ``Definable Algebraic Topology'' from the European Research Council of Martino Lupini. During the late stages of the writing, LC was partially supported by the INdAM ``borse di studio per l’estero'', and is grateful to Yasuhiko Asao and the University of Fukuoka for their kind hospitality.

The authors are thankful to Radmila Sazdanovic and Patrick Martin for a useful discussion that led to the proof of Proposition~\ref{prop:torsion}. The authors are also grateful to the anonymous referee, whose comments led to an improvement of the paper and to Remark~\ref{rem:homotopy}.

\section{Eulerian magnitude homology and regular path homology}
\label{sec:EMH and RPH}

In this section we recall the definitions of magnitude homology~\cite{hepworth2015categorifying}, eulerian magnitude homology~\cite{arXiv:2403.09248}, path homology~\cite{zbMATH06435208} and regular path homology groups~\cite{zbMATH06545866,arXiv:2208.14063} of (directed) graphs. 
In the follow-up, a (directed) graph~$G$ is also considered as a metric space, equipped with the path metric~$d_G$. Unless otherwise stated, all (directed) graphs are finite and connected -- that is, the geometric realization as CW-complex is connected. 

We start with the case of an undirected graph~$G$. For a tuple $(x_0,\dots,x_k)$ of vertices of~$G$, with $x_i\neq x_{i+1}$ and $d_{G}(x_i,x_{i+1})<\infty$ for each~$i$, the \emph{length} of $(x_0,\dots,x_k)$ in $G$ is the number 
\begin{equation}\label{eq:len}
\len(x_0,\dots,x_k)\coloneqq \sum_{i=0}^{k-1}d_{G}(x_i,x_{i+1}) \ .
\end{equation} 
Let~$R$ be a commutative unital ring,  and $l,k\in \mathbb{N}$ natural numbers. We let
    \[
    \MC_{k,l}(G;R)\coloneqq R\langle (x_0,\dots,x_k) \mid x_0\neq \dots\neq x_k, \len(x_0,\dots,x_k)=l \rangle 
    \]
    be the free $R$-module on the tuples of length $l$ consisting of $k+1$ vertices  of $G$. 
    The differential
    $
    \partial=\partial_{k,l}\colon\MC_{k,l}(G;R)\to\MC_{k-1,l}(G;R)
    $
    acts on $(k+1)$-uples $(x_0,\dots,x_k)$ of $G$ by
\begin{equation}\label{def:differential}
    \partial(x_0,\dots,x_k)\coloneqq \sum_{i=1}^{k-1} (-1)^i\partial_i(x_0,\dots,x_k) \ ,
\end{equation}
where the $i$-th face map, defined as 
\[
\partial_i(x_0,\dots,x_k)=
\begin{cases}
    (x_0,\dots,x_{i-1},x_{i+1}, \dots, x_k) & \text{ if } \len(x_0,\dots,x_{i-1},x_{i+1}, \dots, x_k)=l\\
    0 & \text{ otherwise}
\end{cases} 
\] 
forgets the $i$-th entry if the length is preserved, and it is  set to be~$0$ otherwise. The pair $(\MC_{*,l}(G;R), \partial)$ is a chain complex by \cite[Lemma~2.11]{hepworth2015categorifying}. The \emph{magnitude homology}~$\MH_{*,*}(G;R)$ of $G$ is  the bigraded $R$-module $\bigoplus_{k,l\geq 0} \MH_{k,l}(G;R)$, where
\[
\MH_{k,l}(G;R)\coloneqq \mathrm{H}_k(\MC_{*,l}(G;R), \partial)
\]
is  the homology of the magnitude chain complex~\cite[Definition~2.4]{hepworth2015categorifying}. In the following, we will simply write $\MC_{*,*}(G)$ and $\MH_{*,*}(G)$, dropping the reference to the ring $R$ when clear from the context.

\begin{notation}\label{notation:digraphs}
    For an undirected graph $G$, we let $\rho(G)$ be the directed graph obtained from $G$ by replacing each undirected edge $\{x,y\}$ of $G$ with both directed edges $(x,y)$ and $(y,x)$ in $\rho(G)$.
\end{notation}

The definition of magnitude homology extends \textit{verbatim} to directed graphs~\cite{zbMATH07731261}. 
    \begin{remark}\label{rem:MHdirected}
       Note that if $G$ is a directed graph, then the distance $d_G(x,y)$ between vertices of $G$ is the minimum length of directed paths from $x$ to $y$ in $G$. Hence, in the directed case, $\MC_{k,l}(G;R)$ is the $R$-module freely generated by tuples of vertices of $G$ of length $l$, with the property that, for all $i=0,\dots,k-1$, we have $x_i\neq x_{i+1}$, and there exists a directed path in $G$ from $x_i$ to $x_{i+1}$.   If $\rho(G)$ is the directed graph obtained from the undirected graph $G$, then we have $\MH_{k,l}(G)=\MH_{k,l}(\rho(G))$ -- showing that there is no ambiguity in the definition.
    \end{remark}

In the definition of magnitude homology we have assumed that a tuple $(x_0,\dots,x_k)$  of vertices of $G$ is allowed if it satisfies $x_i\neq x_{i+1}$ for all $i=0,\dots,k-1$. A variation of such definition asks that the vertices appearing in the tuples are \emph{all} distinct:

\begin{definition}[{\cite[Definition~3.1]{arXiv:2403.09248}}]
Let $G$ be a graph on the set of vertices~$V(G)$. The $(k,\ell)$-eulerian magnitude chain $\EMC_{k,\ell}(G)$ is the free $R$-module generated by tuples $(x_0,\dots,x_k) \in V(G)^{k+1}$ of length~$\ell$ such that $x_i \neq x_j$ for every $0\leq i,j \leq k$.
\end{definition}
	
It follows from the definition that the eulerian magnitude chain complex is trivial when the length of the tuple is too short to support the necessary landmarks. Indeed, we have the following straightforward property;
	
\begin{lemma}
\label{lem:LowerTriangular}
Let $G$ be an undirected graph, and $k > \ell$ non-negative integers. Then, $\EMC_{k,\ell}(G) = 0$.
\end{lemma}
	
The magnitude differential of Eq.~\eqref{def:differential} restricts to a differential on the eulerian magnitude chains. For a non-negative integer $\ell$, we obtain the \emph{eulerian magnitude chain complex}~$\EMC_{*,\ell}(G)$.
	
\begin{definition}[{\cite[Definition~3.2]{arXiv:2403.09248}}]
\label{def_EMH}
The $(k,\ell)$-eulerian magnitude homology group of an undirected graph $G$ is defined by
$
\EMH_{k,\ell}(G) \coloneqq \mathrm{H}_k(\EMC_{*,\ell}(G),\partial)$. 
\end{definition}

As shown in Remark~\ref{rem:MHdirected}, also eulerian magnitude homology extends to directed graphs and the eulerian magnitude homology $\EMH_{k,l}(G)$ of a graph $G$ agrees with the eulerian magnitude homology of the directed graph $\rho(G)$ associated to $G$. 
By construction, we directly have the following proposition.
	
\begin{proposition}
Let $G$ be a (directed) graph with vertex set $V(G)$. Then, for each positive integer $\ell\ge 0$, we have
\[ 
\EMC_{\ast, \ell}(G) =\bigoplus_{a, b\in V(G)} \EMC_{\ast, \ell}(a, b) \ ,
\]
where $\EMC_{\ast, \ell}(a, b)$ is the subcomplex of $\EMC_{\ast, \ell}(G)$ generated by tuples that start at~$a$ and end at $b$.
\end{proposition}

Although the eulerian magnitude chain complex is defined as a subcomplex of the magnitude chain complex, note that eulerian magnitude homology does not embed in magnitude homology -- see also~\cite{arXiv:2403.09248}.
Furthermore, although the definitions are quite similar, magnitude homology and eulerian magnitude homology do not share the same formal properties; for example, they behave very differently with respect to products. Let $G$ and $H$ be graphs with sets of vertices $V(G)$ and $V(H)$. Recall first that the Cartesian product of $G$ and $H$ is the graph on the vertices $V(G)\times V(H)$ with an edge between $(x_1,y_1)$ and $(x_2,y_2)$ if either $x_1=x_2$ and $\{y_1,y_2\}$ is an edge of~$H$, or $y_1=y_2$ and $\{x_1,x_2\}$ is an edge of $G$. Then,  magnitude homology is known to  satisfy a K\"unneth formula with respect to the Cartesian product of graphs~\cite[Theorem~5.3]{hepworth2015categorifying}. On the other hand, this is not true for eulerian magnitude homology:

\begin{remark}\label{rem:Kunneth}
    Eulerian magnitude homology does not satisfy K\"unneth. In fact, let~$I_2$ be the linear undirected graph with two vertices $x_0,x_1$ and an undirected edge $\{x_0,x_1\}$. Its eulerian magnitude homology is non-trivial only in bidegrees $(k,l)=(0,0),(1,1)$, in which cases it is of rank 2. The Cartesian product $I_2\times I_2$ of~$I_2$ with itself is the cycle graph~$C_4$ on four vertices, which has  non-trivial eulerian magnitude homology in bidegree $(k,l)=(3,5)$. By dimensionality, the generators in bidegree $(3,5)$ can not come from generators of $\EMH_{*,*}(I_2)\otimes \EMH_{*,*}(I_2)$. 
\end{remark}

Despite the lack of formal properties such as a K\"unneth theorem, magnitude homology and eulerian magnitude homology behave quite alike; for example, the behaviour of the torsion is similar. It was first shown by Kaneta and Yoshinaga~\cite{zbMATH07381918}, and then by Sazdanovic and Summers~\cite{zbMATH07370433}, that magnitude homology contains torsion. Analogously,  torsion appears in eulerian magnitude homology as well:

\begin{proposition}\label{prop:torsion}
    Let $A$ be any finitely generated Abelian group. Then, there exists a graph $G$ whose eulerian magnitude homology $\EMH(G;\mathbb{Z})$ contains a subgroup isomorphic to $A$.
\end{proposition}

\begin{proof}
    Inspection of the proof of \cite[Theorems~5.11 \& Corollary~5.12] {zbMATH07381918} (or, equivalently, the results of~\cite[Section~7]{arXiv:2405.04748}), shows that  the homology of a ranked poset embeds in eulerian magnitude homology (of directed graphs). As a consequence, the homology of any pure simplicial complex $K$ embeds in eulerian magnitude homology via the  Kaneta-Yoshinaga construction. Then, the same arguments of  Sazdanovic and Summers~\cite[Theorem~3.14]{zbMATH07370433}  apply, concluding the proof.
\end{proof}

A morphism of (directed) graphs $\phi\colon G_1\to G_2$ is a function $\phi\colon V(G_1)\to V(G_2)$ such that if $e\in E(G_1)$ then $\phi(e)\in E(G_2)$. Hence, a morphism of (directed) graphs   sends (directed) edges to (directed) edges and, in particular, it does not allow collapsing. A morphism of (directed) graphs is called \emph{regular} if it is
injective (as a function). Directed graphs and regular morphisms form a category that we denote by
$\mathbf{Digraph}$. Then, we have the analog of \cite[Prop.~3.3]{hepworth2015categorifying}, for family of regular morphisms of digraphs (opposed to contractions in~\cite[Prop.~3.3]{hepworth2015categorifying}):

\begin{proposition}\label{prop:functoriality}
    Eulerian magnitude homology is a functor on $\mathbf{Digraph}$. 
\end{proposition}

\begin{proof}
    For a regular morphism $f\colon G'\to G$, we let $f_*$ be the  map defined on tuples of $G'$ as 
    \[
    f_*(x_0,\dots,x_k)=
    \begin{cases}
        (f(x_0),\dots,f(x_k))  & \text{ if } \ell(f(x_0),\dots,f(x_k))=\ell(x_0,\dots,x_k)\\
        0 & \text{ otherwise }.
    \end{cases}
    \]
    Since $f$ is injective on the vertices, $f_*$ induces a morphism on the eulerian magnitude chains, hence in eulerian magnitude homology. Furthermore, it respects the identy and compositions of regular morphisms, yielding a functor on $\mathbf{Digraph}$.
\end{proof}

We observe here that the functoriality in Proposition~\ref{prop:functoriality} can not be extended to the family of contractions. To see it, consider the undirected cyclic graph $C_4$ on vertices $\{x_0,x_1,x_2,x_3\}$. 
Consider the contraction $c\colon C_4\to C_3$ of the edge $\{x_2,x_3\}$ of~$C_4$. As $c_*(x_0,x_1,x_2,x_3)=(c(x_0),c(x_1),c(x_2),c(x_3))$, the map $c_*$ can not be non-zero -- as the tuple $(c(x_0),c(x_1),c(x_2),c(x_3))$ would have repetitions.  Then, we must have $c_*(x_0,x_1,x_2,x_3)=0$. Likewise, we must have $c_*(x_0,x_2,x_3)=c_*(x_1,x_2,x_3)=0$. However, this extension to zero does not yield a chain map: 
\[
\partial (c_*(x_0,x_1,x_2,x_3))=0\neq c_*(\partial(x_0,x_1,x_2,x_3))
\]
as $c_*(\partial(x_0,x_1,x_2,x_3))=(c(x_0),c(x_1),c(x_3))-(c(x_0),c(x_1),c(x_2)) $.

We provide some examples of computation of eulerian magnitude homology. 

\begin{example}
    For any directed graph $G$, the rank of $\EMH_{0,0}(G)$ is the number of vertices of $G$, and the rank of $\EMH_{1,1}(G)$ is the number of directed edges in $G$. 
\end{example}

\begin{example}\label{ex:complete}
    The complete undirected graph $K_n$ on $n\geq 1$ vertices has eulerian magnitude homology of rank $$\mathrm{rk}\, \EMH_{k,k}(K_n))=\frac{n!}{(n-(k+1))!} \ , $$ for $k< n$. The groups $\EMH_{k,l}(K_n)$ are trivial for all $k\neq l$.
\end{example}

\begin{example}\label{ex:linear}
     Let $I_n$ be the linear undirected graph on $n$ vertices.  Then, a direct computation shows that the maximum index $\ell$ for which  $\EMH_{k,\ell}(I_n)\neq 0$ is $\ell=\frac{n(n-1)}{2}$, achieved for $k=n-1$.
\end{example}

\begin{example}
Let $G=I^n_2$ be the $n$-fold Cartesian product of the graph $I_2$ with itself. 
A direct computation shows that maximum index $\ell$ for which $\EMH_{k,\ell}(I^n_2)\neq 0$ is the maximum total Hamming distance between pairs of consecutive elements in any permutation of all $2^n$ binary words of length $n$ -- see the \href{https://oeis.org/search?q=1%2C5%2C18%2C53&language=english&go=Search}{A271771} sequence.
\end{example}

\begin{example}
    \label{ex:rank cycles}
    Let $C_n$ be the cycle (undirected) graph on $n$ vertices. Then, the maximum index $\ell$ for which $\EMH_{k,\ell}(C_n)$ is non-trivial is 
    \[
    \ell(n)=\begin{cases}
        \frac{n^2}{2} -n+1 & \text{ if } n \text{ is even}\\
        \frac{(n-1)^2}{2}& \text{ if } n \text{ is odd}
    \end{cases}
    \]
    and it is achieved in bidegree $(n-1,\ell(n))$.
Furthermore,  for every  $(k,\ell)$ it holds that $\EMH_{k,\ell}(C_n)$ is upper bounded by 
    \[
    2\cdot \binom{n}{k+1}\binom{\ell -1}{k} \ ,
    \]
    since we can choose $k+1$ different vertices of the tuple in $\binom{n}{k+1}$ different ways and we can arrange them clockwise or counterclockwise, and the number of ways to distribute $\ell$ steps across $k+1$ different vertices is $\binom{\ell -1}{k}$. 
\end{example}

We conclude the section with recalling the notion of path homology~\cite{zbMATH06435208}, which is strictly related to magnitude homology as proven by Asao~\cite{zbMATH07731261}.

We consider a directed graph $G$ on the vertex set~$V(G)$. For $n\geq 0$, an \emph{elementary path} is an ordered sequence $(x_0,\dots,x_n)$ of vertices of~$G$. We let $\Lambda_n(G)$ be the module freely generated by the $n$-elementary paths of $G$. There is a differential $\partial \colon \Lambda_n(G)\to \Lambda_{n-1}(G) $ as the alternating sums of the face maps forgetting the $i$-th entry of the elementary paths. An elementary path $(x_0,\dots,x_n)$ is called \emph{allowed} if $(x_i,x_{i+1})$ is a directed edge of $G$ for all $i=0,\dots, n-1$. 

\begin{definition}
    An elementary path $(x_0,\dots,x_n)$ is called \emph{regular} if $x_i\neq x_{i+1}$ for all $i=0,\dots n-1$. It is called \emph{strongly regular} if the vertices $x_i$ are all distinct.
\end{definition}

Note that the notion of strongly regular elementary paths was used in~\cite{zbMATH06545866,arXiv:2208.14063} with the name of regular paths. 

Let $\cA_n(G)$ be the module freely generated by the regular allowed elementary $n$-paths of $G$. Observe that the boundary operator $\partial$ does not generally yield a differential on $\cA_*(G)$. Therefore, we consider the submodule $\Omega_n(G)$ of $\cA_n(G)$ on the $\partial$-invariant regular
allowed elementary $n$-paths. Then, $(\Omega_*(G),\partial_*)$ is a chain complex. Its homology groups are called the \emph{path homology} groups of $G$~\cite{zbMATH06435208}. Likewise, we let $\Omega_n^\iota(G)$ be the submodule of  $\cA_n(G)$ on the $\partial$-invariant strongly regular
allowed elementary $n$-paths. The pair $(\Omega_*^\iota(G),\partial_*)$ is also a chain complex~\cite{zbMATH06545866,arXiv:2208.14063}. 

\begin{definition}
    The homology groups of the chain complex $(\Omega_*^\iota(G),\partial_*)$ are the \emph{regular path homology} groups of $G$.
\end{definition}

The condition of taking strongly regular paths has interesting consequences on the associated path homology groups. First, observe that the non-trivial regular path homology groups are always finite. Secondly, Fu and Ivanov showed in \cite{arXiv:2407.17001} that path homology of graphs is not simplicial. On the other hand, regular path homology is the homology of a CW-complex by~\cite{zbMATH06545866}. Furthermore, as remarked in~\cite{arXiv:2208.14063} a Lefschetz fixed point theorem only holds for regular path homology.

\section{Regularly diagonal graphs}
\label{sec:regularly diag graphs}

In this section we focus on the \emph{diagonality} properties of eulerian magnitude homology of undirected graphs, providing a  characterization of complete graphs in terms of  eulerian magnitude homology.

Recall that a graph is called \emph{diagonal} if its magnitude homology is concentrated on the diagonal, that is $\MH_{k,l}(G)=0$ for $k\neq l$~\cite[Definition~7.1]{hepworth2015categorifying}. Diagonality of graphs was shown to be connected to relevant properties of the graphs at hand~\cite{hepworth2015categorifying, bottinelli2020magnitude, asao2024girth, arXiv:2405.04748}. Furthermore, it is preserved under taking Cartesian products of graphs~\cite[Proposition~7.3]{hepworth2015categorifying} and any non-trivial join of graphs is diagonal~\cite[Theorem~7.5]{hepworth2015categorifying}.

\begin{example}\label{ex:treesdiagonal}
    Trees and complete undirected graphs are diagonal.
\end{example}

Recall that the girth of a graph $G$ is the minimum
length of cycles in $G$.

\begin{proposition}[{\cite[Corollary~1.6]{asao2024girth}}]
  Let $G$ be a diagonal graph which is not a tree. Then,   the girth of $G$ is either 3 or 4.
\end{proposition}

The notion of diagonality for graphs extends to digraphs~\cite{arXiv:2405.04748} in the obvious way, and, in analogy with the definition of diagonal (directed) graphs, we can consider graphs with diagonal eulerian magnitude homology. 

\begin{definition}
    A (directed) graph $G$ is called \emph{regularly diagonal} if $\EMH_{k,l}(G)=0$ for all $k\neq l$.
\end{definition}

Some first properties of diagonality easily follow. For a directed graph~$G$, denote by $G^{\op}$ the directed graph with the same vertices of $G$ and with edges of $G$ oriented in the opposite direction.

\begin{lemma}
    A directed graph $G$ is (regularly) diagonal if and only if $G^\op$ is (regularly) diagonal.
\end{lemma}

\begin{proof}
    Consider the map
    \[
    (x_0,\dots,x_k)\mapsto (x_k,\dots, x_0)
    \]
    sending a tuple of $G$ to the tuple in $G^\op$ with same vertices and opposite order. Then, this map induces an isomorphism of magnitude homology groups~\cite[Lemma~6.3]{arXiv:2405.04748}, and, in the same way, an isomorphism of eulerian magnitude homology groups. 
\end{proof}

\begin{lemma}\label{lem:diagvsregdiag}
    Let $G$ be a directed graph without directed cycles. Then, $G$ is diagonal if and only if it is regularly diagonal.  
\end{lemma}

\begin{proof}
    If $G$ has no directed cycles, then the eulerian magnitude chain complex of $G$ is the magnitude chain complex of $G$, and the result follows.
\end{proof}

\begin{remark}
    If $G$ has non-trivial diagonal (eulerian) 
    magnitude homology only in bidegree $(k,k)=(0,0)$ then $G$ is the totally disconnected graph.
\end{remark}

By Lemma~\ref{lem:diagvsregdiag}, the family of regularly diagonal graphs which are not diagonal have to include graphs with directed cycles. 
Diagonality is strictly related to the presence of cycles, and  the next result, which follows from \cite[Theorem~1.5]{asao2024girth}, shows that the bigger the girth of the graph, the further from the diagonal the eulerian magnitude homology can be concentrated. 

\begin{lemma}\label{lem:regdiagandgirth}
    Let $G$ be an undirected regularly diagonal graph. Then the girth $\mathrm{gir}(G)$ of $G$ is either $3$ or $4$.
\end{lemma}

We shall see that Lemma~\ref{lem:regdiagandgirth} can be improved, and in Corollary~\ref{cor:regdiagcompletegraphs} we provide a complete characterization of regularly diagonal undirected graphs. 
To proceed with such characterization, we  need to analyze the relation between the existence of cycles in an undirected graph and its eulerian magnitude homology. First,  we recall the following fact -- see \cite[Lemma~3.2]{arXiv:2403.09248}:

\begin{remark}\label{rem:cyclesEMH22}
    Let $G$ be a graph without 3- and 4-cycles. Then, $\EMH_{2,2}(G)=0$.
\end{remark}

Following the arguments of~\cite[Theorem~4.3]{zbMATH07370433}, we can extend the result in Remark~\ref{rem:cyclesEMH22} to the whole diagonal:

\begin{proposition}\label{prop:vanishingdiagonal}
    Let $G$ be an undirected graph without 3- and 4-cycles. Then, $\EMH_{k,k}(G)=0$ for all $k\geq 2$.
\end{proposition}

\begin{proof}
    Let $k\geq 2$ be a natural number. As $\EMC_{k+1,k}(G)$ is always trivial,  $\EMH_{k,k}(G)$ is the kernel of the differential
    \[
    \EMC_{k,k}(G)\to\EMC_{k-1,k}(G) \ . 
    \]
    Recall that $\EMC_{k,k}(G)$ is generated by tuples $(x_0,\dots,x_k)$ of distinct vertices of~$G$ with $\len(x_0,\dots,x_k)=k$. This means that  $d_G(x_i,x_{i+1})=1$ for all $i=0,\dots, k-1$. 
  
    Assume that  a tuple $(x_0,\dots,x_k)$ is in the kernel of the differential. This happens if and only if $\len(x_0,\dots,\hat{x_i},\dots,x_k)<k$ for all $i$. However, if for some $i$ the length of the tuple is not preserved, we have that the triple $(x_{i-1},x_i,x_{i+1})$ forms a cycle. By assumption, $G$ contains no 3-cycles, and this can not happen. As a consequence, no tuple of type $(x_0,\dots,x_k)$ with $k\geq 2$ is in the kernel of the differential. 

    We now need to show that there are no linear combinations of tuples in $\EMC_{k,k}(G)$ with trivial differential. Let $\gamma$ be a non-trivial linear combination of tuples in $\EMC_{k,k}(G)$ with $\gamma $ in the kernel of the differential. Let $(x_0,\dots,x_k)$ be a tuple appearing in the combination $\gamma$ with scalar $c\neq 0$. We know that for some $i$, the distance is $d_G(x_{i-1},x_{i+1})=2$. As $\gamma$ is in the kernel of the differential, the tuple $(x_0,\dots,\hat{x_i},\dots,x_k)$ must be canceled in $\EMC_{k-1,k}(G)$ by another tuple of type $\partial_i(x_0,\dots,x_{i-1},y,x_{i+1},\dots, x_k)$ with $(x_0,\dots,x_{i-1},y,x_{i+1},\dots, x_k)$ also appearing in $\gamma-c(x_0,\dots,x_{i-1},x_i,x_{i+1},\dots, x_k)$. But then the vertices $x_{i-1}, x_i, x_{i+1},y$ form a 4-cycle. This is not possible by assumption and, as a consequence, the kernel of the differential is always 0 for $k\geq 2$. 
\end{proof}

Observe that the requirement not to have  3- or 4-cycles is strict, as the following counterexamples show:

\begin{example}
    Let $G=C_3$ be the cycle graph on 3 vertices. Then, $\EMH_{2,2}(G)$ has rank 6. If $G=C_4$ the cycle graph with 4 vertices, then $\EMH_{2,2}$ has rank 4. 
\end{example}

Recall from Example~\ref{ex:treesdiagonal} that undirected trees are diagonal graphs, and, as trees have no cycles, the eulerian magnitude homology in bidegree $(k,k)$, $k\geq 2$, is trivial by Proposition~\ref{prop:vanishingdiagonal}. One may be tempted to infer that trees are also regularly diagonal; however, this is not true, as the next example shows:

\begin{example}
    Let $I_n$ be the linear undirected graph on $n$ vertices. Then, Example~\ref{ex:linear} shows that eulerian magnitude homology of $I_n$ in bidegree $\left(n-1,\frac{n(n-1)}{2}\right)$ is non-trivial.
\end{example}

We now give a complete characterization of undirected graphs which are regularly diagonal.

\begin{theorem}\label{thm:charregdiag}
    Let $G$ be a finite connected regularly diagonal undirected graph, with eulerian magnitude homology groups $\EMH_{k,k}(G)=0$  in all bidegrees $(k,k)$ for $k>n$. 
    Then, $G$ is the complete graph on $n+1$ vertices. 
\end{theorem}

\begin{proof}
    Let $n+1>2$ be the number of vertices of $G$, the case on $1$ and $2$ vertices being straightforward. For each $x$ in $G$ denote by $\ell(x)$ the maximum distance achievable from $x$; that is
    \[
    \ell(x)\coloneqq \max\{d_G(x,y)\mid y\in V(G)\}\geq 1 \ .
    \]
    We want to show that $\ell(x)=1$ for all $x$ in $G$.     Assume this is not the case and let~$D$ be the diameter of $G$, that is, $D=\max\{\ell(x)\mid x\in V(G)\}\geq 2$. We now  iteratively construct a tuple $(x_0,\dots,x_n)$ consisting of all possible vertices of $G$,  as follows:
    \begin{itemize}
        \item pick any vertex $x_0$  of $G$ with $\ell(x_0)=D$; 
        \item having defined the tuple $(x_0,\dots,x_{i-1})$, if $i-1\neq n$, we choose $x_i$ such that $d_G(x_{i-1},x_i)=\max \{d_G(x_{i-1},y)\mid y\in V(G)\setminus \{x_0,\dots,x_{i-1}\}\} $. 
    \end{itemize}
    Observe that, by construction,  the length~$l$ of the tuple $(x_0,\dots,x_n)$ is greater than $n$.
    As $(x_0,\dots,x_n)$ is  of maximal size, it is not in the image of the differential. We want   to show that it is a cycle, concluding the proof.  
    To show that  the tuple $(x_0,\dots,x_n)$ is a cycle, it is enough to show that all subtuples $(x_0,\dots,\hat{x_i},\dots , x_n)$ have length smaller than $(x_0,\dots, x_n)$. For a series of vertices $x_{i-1}, x_i, x_{i+1}$ in the tuple, observe that we always have $d_G(x_{i-1},x_{i+1})\leq d_G(x_{i-1},x_{i})+d_G(x_{i},x_{i+1})$ by the triangle inequality. In addition, $d_G(x_{i-1},x_{i+1})\leq \ell(x_{i-1})$ by definition of $\ell(x_{i-1})$, and, by construction of the tuple $(x_0,\dots,x_n)$, we have that $d_G(x_{i-1},x_i)$ is maximal among the distances $d_G(x_{i-1},y)$ with $y\in V(G)\setminus\{x_0,\dots,x_{i-1}\}$. Hence, $d_G(x_{i-1},x_i)\geq d_G(x_{i-1},x_{i+1})$. As a consequence, we have $d_G(x_{i-1},x_{i+1})< d_G(x_{i-1},x_{i})+d_G(x_{i},x_{i+1})$; that is, $\len(x_0,\dots,\hat{x_i},\dots,x_n)<\len(x_0,\dots,x_n)$ and $(x_0,\dots,x_n)$ is a cycle in $\EMH_{n,l}(G)$.  However, this is not possible because $l>n$ and $G$ is regularly diagonal. 
\end{proof}

Recall from Example~\ref{ex:complete} that complete graphs are regularly diagonal.

\begin{corollary}\label{cor:regdiagcompletegraphs}
    Eulerian magnitude homology detects complete graphs.
\end{corollary}

The results exhibited so far in this section reveal a connection between the girth of a graph and the diagonality of its eulerian magnitude homology.
In particular, the intuition supported by Theorem \ref{thm:charregdiag} is that graphs with small girth exhibit diagonality.
This aligns with the intuition that the presence of short cycles,  enforcing a more rigid pattern, simplifies the combinatorial structure of eulerian magnitude homology.
In contrast, as the girth increases, deviations from diagonality emerge.

\begin{lemma}
    \label{lem:subdiagonal given girth}
    Let $G$ be an undirected graph with girth $\gir(G)=\gamma$.
    Then $G$ has non-trivial eulerian magnitude homology until at least the $\lceil \frac{(\gamma-1)^2}{2} \rceil - (\gamma -1)$ subdiagonal.
\end{lemma}

\begin{proof}
    If $\gir(G)=\gamma$ then $G$ contains $C_\gamma$ as a subgraph with vertex set $\{x_1,\dots,x_{\gamma}\}$.
    So, it is possible to construct a tuple of vertices of type $c=(x_1,x_{\lfloor \frac{\gamma}{2}\rfloor},x_{\gamma},x_{\lfloor \frac{\gamma}{2}\rfloor -1},\dots)$, maximizing at each step the distance between consecutive elements in $c$. Hence, by construction $c$ is a non-trivial cycle.
    Then, as in  Example~\ref{ex:rank cycles}, the statement follows.
\end{proof}

\begin{remark}
    Notice that the fact that two graphs $G$ and $H$ have the same girth, does not imply that they have non-trivial eulerian magnitude homology to the same subdiagonal. Indeed, take $G$ to be $K_4$ with one edge removed. Then, $\gir(G)=\gir(K_4)$ but $G$ has $\EMH_{2,3}(G)\neq \emptyset$. 
\end{remark}

At this stage, having explored diagonality in $\EMH$ and its relationship to graph girth, a natural question arises: how does $\EMH(G)$ behave when $G$ is not the complete graph? Specifically, to what extent does $\EMH(G)$ deviate from the diagonal in such cases?
To address this, we introduce a distance metric between graphs based on graph density, which quantifies how far a given graph $G$ on $n$ vertices is from the complete graph $K_n$.
We then establish a result connecting this distance to the deviation of $\EMH(G)$ from the diagonal, providing a measure of how the homology of $G$ diverges from that of $K_n$ as the graph structure becomes sparser.

\begin{definition}
\label{def:subgraph network}
    Let $G$ be a graph. The \emph{subgraph network associated to~$G$}, denoted by $2^G$, is the undirected graph such that:
    \begin{itemize}
        \item the set of vertices $V(2^G)$ of $2^G$ is the set of all non-isomorphic connected spanning subgraphs of $G$;
        \item there is an edge between two vertices of $2^G$, say $H_1$ and $H_2$, if and only if for every vertex $v \in V(H_1)\cap V(H_2)$ it holds that $|\deg_{H_1}(v)-\deg_{H_2}(v)| \leq 1$.
    \end{itemize}
\end{definition}

As the vertices of $2^G$ can be represented by subgraphs of $G$, we say that $H$ is contained in $2^G$ if it appears as a vertex of $2^G$. 
    We provide in Figure \ref{fig:subgraph network K4} the subgraph network of the complete graph $K_4$.
    \begin{figure}[h]
        \begin{tikzpicture}[node distance={12mm}, thick, main/.style = {draw, circle}]

            \node[main] (0) {$0$}; 
		  \node[main] (1) [right of=0] {$1$};  
		  \node[main] (2) [above of=1] {$2$};
		  \node[main] (3) [above of=0] {$3$};
		  \draw (0) -- (1);
		  \draw (1) -- (2);
		  \draw (0) -- (2);
		  \draw (2) -- (3);
            \draw (0) -- (3);
            \draw (1) -- (3);
            \node[circle, draw=red, fit=(0) (1) (2) (3), inner sep=.1mm] (K4) {};

            \node[main] (2') [below left of=0] {$2$};
            \node[main] (3') [left of=2'] {$3$};
            \node[main] (0') [below of=3'] {$0$};
            \node[main] (1') [below of=2'] {$1$};
            \draw (0') -- (1');
		  \draw (1') -- (2');
		  \draw (0') -- (2');
		  \draw (2') -- (3');
            \draw (0') -- (3');
            \node[circle, draw=red, fit=(0') (1') (2') (3'), inner sep=.1mm] (K4') {};
            \draw[red] (K4) -- (K4');

            \node[main] (3'') [below right of=1] {$3$};
            \node[main] (2'') [right of=3''] {$2$};
            \node[main] (0'') [below of=3''] {$0$};
            \node[main] (1'') [below of=2''] {$1$};
            \draw (0'') -- (1'');
		  \draw (1'') -- (2'');
		  \draw (0'') -- (2'');
		  \draw (1'') -- (3'');
            \draw (0'') -- (3'');
            \node[circle, draw=red, fit=(0'') (1'') (2'') (3''), inner sep=.1mm] (K4'') {};
            \draw[red] (K4) -- (K4'');
            \draw[red] (K4') -- (K4'');

            \node[main] (2''') [below left of=0'] {$2$};
            \node[main] (3''') [left of=2'''] {$3$};
            \node[main] (0''') [below of=3'''] {$0$};
            \node[main] (1''') [below of=2'''] {$1$};
            \draw (0''') -- (1''');
		  \draw (1''') -- (2''');
		  \draw (2''') -- (3''');
            \draw (0''') -- (3''');
            \node[circle, draw=red, fit=(0''') (1''') (2''') (3'''), inner sep=.1mm] (K4''') {};
            \draw[red] (K4') -- (K4''');
            \path[every node/.style={font=\sffamily\small}]
			(K4) edge [bend right=50, red] node {} (K4''');

            \node[main] (3'''') [below right of=1'] {$3$};
            \node[main] (2'''') [right of=3''''] {$2$};
            \node[main] (0'''') [below of=3''''] {$0$};
            \node[main] (1'''') [below of=2''''] {$1$};
            \draw (0'''') -- (1'''');
		  \draw (1'''') -- (2'''');
		  \draw (1'''') -- (3'''');
            \draw (0'''') -- (3'''');
            \node[circle, draw=red, fit=(0'''') (1'''') (2'''') (3''''), inner sep=.1mm] (K4'''') {};
            \draw[red] (K4') -- (K4'''');
            \draw[red] (K4'') -- (K4'''');
            \draw[red] (K4''') -- (K4'''');

            \node[main, node distance={20mm}] (2''''') [below of=1'''] {$2$};
            \node[main] (3''''') [left of=2'''''] {$3$};
            \node[main] (0''''') [below of=3'''''] {$0$};
            \node[main] (1''''') [below of=2'''''] {$1$};
            \draw (0''''') -- (1''''');
		  \draw (1''''') -- (2''''');
            \draw (0''''') -- (3''''');
            \node[circle, draw=red, fit=(0''''') (1''''') (2''''') (3'''''), inner sep=.1mm] (K4''''') {};
            \draw[red] (K4''') -- (K4''''');
            \draw[red] (K4'''') -- (K4''''');
            \path[every node/.style={font=\sffamily\small}]
			(K4'') edge [bend left=90, looseness=1.2, red] node {} (K4''''');

            \node[main, node distance={20mm}] (2'''''') [below of=1''''] {$2$};
            \node[main] (3'''''') [left of=2''''''] {$3$};
            \node[main] (0'''''') [below of=3''''''] {$0$};
            \node[main] (1'''''') [below of=2''''''] {$1$};
            \draw (0'''''') -- (1'''''');
		  \draw (1'''''') -- (2'''''');
            \draw (1'''''') -- (3'''''');
            \node[circle, draw=red, fit=(0'''''') (1'''''') (2'''''') (3''''''), inner sep=.1mm] (K4'''''') {};
            \draw[red] (K4''') -- (K4'''''');
            \draw[red] (K4'''') -- (K4'''''');
            \draw[red] (K4''''') -- (K4'''''');

        \end{tikzpicture}
        \vspace*{-2cm}
        \caption{Subgraph network of $K_4$. The network $2^{K_4}$ is represented in red, and the non-isomorphic spanning connected subgraphs of $K_4$ defining the vertices of $2^{K_4}$ are drawn in black.}
        \label{fig:subgraph network K4}
    \end{figure}

The following result is a straightforward consequence of Definition \ref{def:subgraph network}.

\begin{lemma}
\label{lemma:network connected, diameter}
    Let $G$ be a connected undirected graph. Then, we have:
    \begin{enumerate}
        \item $2^G$ is connected;
        \item the diameter of $2^G$ is $\max_{v\in V(G)} \deg_G(v)$.
    \end{enumerate}
\end{lemma}

Let $G$ and $H$ be undirected graphs with  $n$ vertices.
       We further define the \emph{covering subgraph network of $G$ and $H$}, denoted by $2^{(G,H)}$, as the subgraph network associated to the complete graph $K_n$.

\begin{example}
    Let $G$ be the linear graph~$I_4$ on $4$ vertices, and let $H$ be the bipartite graph $K_{1,3}$.
    Then, the covering subgraph network $2^{(I_4,K_{1,3})}$ is the subgraph network $2^{K_4}$ represented in Figure \ref{fig:subgraph network K4}.
\end{example}

We are now ready to introduce a metric between graphs with same number of vertices. To do it, recall from Lemma~\ref{lemma:network connected, diameter} that the graph $2^{(G,H)}$ is connected.

\begin{definition}
Let $G,H$ be undirected graphs with $|V(G)|=|V(H)|$.
    The distance $\Delta$ between $G$ and $H$,   
    is the distance $\Delta(G,H)$ between $G$ and $H$ considered as vertices of $2^{(G,H)}$:
    \[
    \Delta(G,H)\coloneqq d_{2^{(G,H)}}(G,H) \ ,
    \]
    where $d_{2^{(G,H)}}$ is the path metric on $2^{(G,H)}$.
\end{definition}

\begin{example}
    The distance $\Delta(K_{1,3},C_4)$ between the bipartite graph $K_{1,3}$ and the cycle graph $C_4$ is $\Delta(K_{1,3},C_4)=1$, as shown in Figure \ref{fig:subgraph network K4}. 
\end{example}

\begin{notation}
    \label{notation:maximal subcycle}
    For the undirected simple complete graph $K_n$, we call $\gamma_n^s$ the girth of the maximal sub-cycle achievable by removing $s$ edges from $K_n$, see Figure \ref{fig:gamma 34 10} for an example.
\end{notation}

\begin{figure}[h]
    \begin{tikzpicture}[every node/.style={circle, draw, minimum size=5mm}]

    \foreach \x in {0,...,9}
    \node (\x) at (162-\x*36:2cm) {\x};

    \foreach \u/\v in {0/1, 1/2, 2/3, 3/4, 4/9, 0/9, 4/5, 4/6, 4/7, 4/8,
                   5/6, 5/7, 5/8, 5/9, 6/7, 6/8, 6/9, 7/8, 7/9, 8/9}
    \draw (\u) -- (\v);

\end{tikzpicture}
\caption{By removing $34$ edges from $K_{10}$ we can build a $C_6$, so we have $\gamma_{10}^{34}=6$.}.
\label{fig:gamma 34 10}
\end{figure}

With these preliminaries in place, we are now ready to state the result connecting $\Delta(G,K_n)$ to the deviation of $\EMH(G)$ from the diagonal.

\begin{proposition}
\label{prop:min diagonal if sparse enough}
    Let $G$ be a graph on $n$ vertices such that $\Delta(G,K_n)=h$ and $C_\gamma$, with $\gamma \geq \gamma^h_n$, is a cyclic subgraph of $G$.
    Then $G$ has non-trivial eulerian magnitude homology until at least subdiagonal $\lceil\frac{(\gamma-1)^2}{2}\rceil - (\gamma -1)$.
\end{proposition}

\begin{proof}
    The proof is a direct application of Lemma \ref{lem:subdiagonal given girth}.
\end{proof}

    The hypothesis in Proposition \ref{prop:min diagonal if sparse enough} that $G$ contains a big enough cycle $C_\gamma$ is crucial: it is not enough for the graph $G$ to be globally far from the complete graph in terms of edge count -- the \emph{distribution} of missing edges plays a critical role:
\begin{example}
    Consider the graph $G = K_{n-1} \cup \{e\}$, where $e$ is a single edge connecting one isolated vertex $v \notin K_{n-1}$ to a vertex in $K_{n-1}$. Then $\Delta(G,K_n)=n-2$, since it lacks all but one of the $n-1$ edges incident to $v$. Nevertheless, a direct computation shows that $\EMH(G)$ is non-trivial only until the first subdiagonal.
\end{example}
    This example illustrates that a large edge deficit is not sufficient, on its own, to force non-trivial $\EMH(G)$ strictly below the first subdiagonal: the missing edges must induce a ``structurally sparse" region in the graph, such as a large cycle.
\\

We end this section by establishing a result linking the diagonality of eulerian magnitude homology to a structural decomposition in standard magnitude homology. Specifically, we show that when $G$ is regularly diagonal then eulerian magnitude homology gives a splitting in magnitude homology in bidegree $(k,k)$.

Since the eulerian magnitude chain complex is a subcomplex of the standard magnitude chain complex, we can naturally derive a long exact sequence connecting the two.
Given that the generators of the eulerian magnitude chain groups form a subset of those in the usual magnitude chains, describing the quotient is fairly straightforward: it is generated by $k$-trails that revisit at least one vertex. 
This quotient group was introduced in \cite{arXiv:2403.09248} as the \emph{discriminant $(k,\ell)$-magnitude chain group}
\[
\DMC_{k,\ell}(G)=\frac{\MC_{k,\ell}(G)}{\EMC_{k,\ell}(G)}\ .
\]
By definition of $\EMC_{k,\ell}(G)$ and $\DMC_{k,\ell}(G)$ we have the usual short exact sequence of chain complexes
\[
0 \to \EMC_{\ast,\ell}(G) \xrightarrow{\iota} \MC_{\ast,\ell}(G) \xrightarrow{\pi} \DMC_{\ast,\ell}(G) \to 0 \ ,
\]
where $\iota$ and $\pi$ are the induced inclusion and quotient maps, respectively. Therefore, we obtain a long exact sequence in homology
\begin{equation}
	\label{eq:LES}
	\cdots \to 0 \to \EMH_{k,k}(G) \xrightarrow{\iota_*} \MH_{k,k}(G) \xrightarrow{\pi_*} \DMH_{k,k}(G) \xrightarrow{\delta_{k}} \EMH_{k-1,k}(G) \to \cdots,
\end{equation}
where the map $\delta_{k+1}$ is given by the Snake Lemma.

\begin{lemma}
    Let $G$ be a regularly diagonal graph. Then,
    \[
    \MH_{k,\ell}(G)=\begin{cases}
    \EMH_{k,\ell}(G) \oplus \DMH_{k,\ell}(G) &\text{ if } k = \ell, \\
    \DMH_{k,\ell}(G) &\text{ otherwise.}
\end{cases}
\]
\end{lemma}

\begin{proof}
If $G$ is regularly diagonal, then $\EMH_{k,\ell}(G)$ is trivial for every $k\neq \ell$. As 
$\DMH_{\ell,\ell}(G)$ is free, the result follows.
\end{proof}

\section{Regularly diagonal directed graphs}
\label{sec:regularly diag digraphs}

We have seen that eulerian magnitude homology of undirected graphs detects the complete graphs. The story is completely different when we consider the bigger family of directed graphs. In this section we show that taking cones and joins preserves diagonality, and as an example we compute the (eulerian) magnitude homology of transitive tournaments.  
 
First, observe that the simplest family of directed graphs which are regularly diagonal, but not of the form $\rho(K_n)$, is the family of alternating digraphs. For an undirected graph~$G$ and an orientation $o$ of its edges, we say that $o$ is alternating  if there exists a partition $V \sqcup
W$ of the vertices of $G$ such that all elements of $V$ have indegree~$0$ and all elements of $W$ have outdegree~$0$ -- \textit{cf}.~\cite{zbMATH07928744}. The directed graph $(G,o)$ is the undirected graph $G$ equipped with the orientation $o$ of its edges.
Then, a  direct computation shows the following:

\begin{lemma}\label{lemma:alternating}
    Let $G$ be an undirected graph and $o$ an alternating orientation. Then, 
    \[
    \mathrm{rk}\,\MH_{k,l}(G,o)=
\mathrm{rk}\,\EMH_{k,l}(G,o)=
\begin{cases}
    |V(G)| & \text{ if } (k,l)=(0,0)\\
    |E(G)| & \text{ if } (k,l)=(1,1)\\
    0 & \text{ otherwise } 
\end{cases}
\]
  where $V(G)$ and $E(G)$ denote the vertices and edges of $G$. In particular,   the directed graph $(G,o)$ is (regularly) diagonal. 
\end{lemma}

Observe that the existence of an alternating orientation on $G$ implies that $G$ is a bipartite graph:

\begin{remark}
    Let $G$ be a connected graph.
    Notice that it is possible to construct the directed graph $(G,o)$ if and only if $G$ is bipartite.
    Indeed, if $G=(V,E)$ is bipartite then $V=V_1 \cup V_2$ with $V_1\cap V_2 = \emptyset$, and  it is possible to choose an orientation $o$ of the edges such that all vertices in $V_1$ have indegree $0$ and all elements of $V_2$ have outdegree $0$.
    On the other hand, say we have a directed graph $(G,o)$ and suppose $G$ is not bipartite. 
    Then it not possible to divide the vertex set $V$ of $G$ in two disjoint independent sets, meaning for every choice of $V_1,V_2$ such that $V_1 \cup V_2 = V$ there exists a vertex $v \in V_1 \cap V_2$.
    But then by construction of $(G,o)$ the vertex $v$ would have both indegree and outdegree equal to $0$, contradicting the connectivity of $G$. 
\end{remark}

The following example shows that not only alternating graphs have eulerian magnitude homology concentrated in bidegrees $(0,0)$ and $(1,1)$; but also directed linear graphs.

\begin{example}\label{ex:lingraphs}
    Let $L_n$ be the directed linear graph on $n$ vertices $1,\dots,n$ and directed edges of type $(i,i+1)$. Then, (eulerian) magnitude homology of $L_n$ is non-trivial only in bidegrees $(0,0)$ and $(1,1)$. In fact, as the directed graph $L_n$ has only edges of type $(i,i+1)$, any tuple $(i_1,\dots,i_k)$ of length~$\ell$ can be completed to the full tuple $(i_1,i_1+1,\dots, i_{k-1}, i_k)$, which has the same length~$\ell$, and it is the only admissible tuple of~$\ell+1$ vertices of $L_n$ starting at $x_1$ and ending at $x_k$; in particular, we have $x_k=x_1+\ell$. Observe that the set of tuples of vertices of $L_n$ starting at $x_1$ and ending at $x_k$ is partially ordered by inclusion, and the resulting poset is a Boolean poset. 
    Furthermore, for a given length~$\ell$, the chain complex $\EMH_{*,\ell}(L_n)$ decomposes as the direct sum
    \[
    \EMH_{*,\ell}(L_n)=\bigoplus_{|x_1-x_k|=\ell}\EMH_{*,\ell}(x_1,x_k)
    \]
where the sum is taken    across all vertices $x_1<x_k$ of $L_n$ at distance $\ell$; here $\EMH_{*,\ell}(x_1,x_k) $ denotes the sub-chain complex consisting of all chains starting at~$x_1$ and ending at $x_k$. Now, observe that each such term $\EMH_{*,\ell}(x_1,x_k) $ is the homology of a Boolean poset, which is trivial. As a consequence, for any $\ell>1$, the (eulerian) magnitude homology groups $\EMH_{k,\ell}(L_n)$ are trivial. It is left to consider the case of bidegrees $(0,0)$ and $(1,1)$. In such cases, the (eulerian) magnitude homology groups are of rank $n$ and $n-1$, respectively.
\end{example}

As a consequence of Lemma~\ref{lemma:alternating}, all bipartite graphs can be made regularly diagonal with the choice of an alternating orientation. Furthermore, using the arguments of Example~\ref{ex:lingraphs}, we have the following:

\begin{proposition}
Let $T$ be a directed tree. Then, $T$ is regularly diagonal.    
\end{proposition}

The class of bipartite graphs and  trees provides many non-trivial examples of regularly diagonal digraphs. 
We shall also construct regularly diagonal graphs by means of taking cones and joins. First, recall that   the cone $C(G)$ of a digraph $G$ is defined as follows. The vertices of $C(G)$ are the vertices of $G$ with an extra vertex~$\star$ added. The directed edges of $C(G)$ are the directed edges of $G$, together with all directed edges of type $(x,\star)$ for all $x$ in $G$. 

\begin{proposition}\label{prop:cones}
    Let $C(G)$ be the cone of $G$. Then, we have:
    \[
    \EMH_{k,\ell}(C(G))\cong 
    \begin{cases}
        \EMH_{k,\ell}(G)\oplus\mathbb{Z} & \text{ if } (k,\ell)=(0,0) ;\\
        \EMH_{k,\ell}(G)\oplus \EMH_{k-1,\ell-1}(G) & \text{ otherwise } .
    \end{cases}
    \]
\end{proposition}

\begin{proof}
    As usual, $\EMH_{0,0}(C(G))$ computes the number of vertices of $C(G)$, which is the number of vertices of $G$ plus the extra vertex $\star$. 

    Let $(k,\ell)\neq (0,0)$ be any other bidegree. First, observe that $\EMC_{k,\ell}(C(G))$ consists either of tuples $(x_0,\dots,x_k)$ of distinct vertices of $G$ of length $\ell$, or of tuples of type $(x_0,\dots,x_{k-1},\star)$ with $(x_0,\dots,x_{k-1})$ of length $\ell-1$. Hence, we have an isomorphism 
    \begin{equation}\label{eq:splitcone}
        \EMC_{k,\ell}(C(G))\cong\EMC_{k,\ell}(G)\oplus \EMC_{k-1,\ell-1}(G)
    \end{equation} 
    of modules:  an element of  $\EMC_{k,\ell}(C(G))$ either corresponds to a tuple $(x_0,\dots,x_k)$ in $G$, or to a tuple of type  $(x_0,\dots,x_{k-1},\star)$. The differential on the first component is the differential of $\EMC_{*,*}(G)$. Consider the second component. In such case, we have $\partial_i(x_0,\dots,x_{k-1},\star)=(x_0,\dots,\hat{x_{i}},\dots,x_{k-1},\star)$ for all $i=0,\dots,k_1$. The last face map~$\partial_k$ sends $(x_0,\dots,x_{k-1},\star)$ to $0$, as $(x_0,\dots,x_{k-1})$ is of length $\len(x_0,\dots,x_{k-1},\star)-1$. Hence, also the differential on $\EMC_{k,\ell}(C(G))$ is consistent with the differential on the decomposition $\EMC_{k,\ell}(G)\oplus \EMC_{k-1,\ell-1}(G)$; showing that Eq.~\ref{eq:splitcone} is in fact an isomorphism of chain complexes. This is enough to show that $\EMH_{k,\ell}(C(G)$ is isomorphic to $\EMH_{k,\ell}(G)\bigoplus \EMH_{k-1,\ell-1}(G)$.
\end{proof}

Consequently, directed cones of regularly diagonal graphs are regularly diagonal. 

\begin{example}
    Let $C(I_2)$ be the cone graph over the bidirected graph $\rho(I_2)$. That is, $C(I_2)$ has vertices $\{x_0,x_1,\star\}$ and directed edges $(x_0,x_1), (x_1,x_0), (x_0,\star), (x_1,\star)$. Then, $C(I_2)$ is regularly diagonal, with $\EMH_{2,2}(C(I_2))$ of rank~2. Observe also that the same holds true for the opposite cone, which is the digraph on the same vertices but with directed edges $(x_0,x_1), (x_1,x_0), (\star,x_0), (\star,x_1)$.
\end{example}

Recall that a transitive tournament is an orientation of the complete graph without directed cycles, constructed so that there is a directed edge between every pair of vertices and the edges follow a transitive order, meaning that if we call the vertex set $V=\{0,\dots,n\}$ there is a directed edge $i \to j$ if and only if $i < j$. We denote by~$T_n$ the transitive tournament on $n+1$ vertices; see Figure~\ref{fig:transitive tournament} for an example.

\begin{figure}[h]
\begin{center}
\begin{tikzpicture}[every node/.style={inner sep=1.5em}]
    \SetGraphUnit{2}
    \renewcommand*{\VertexLineColor}{black}
    \SetVertexNoLabel 
    \Vertices{circle}{1,2,3,4,5,6}
    \SetUpEdge[style={shorten >= 1pt,black, shorten <= 1pt,->,thick}]
    
     \foreach \v [count=\vi from 2] in {1,...,5}
     {
     \foreach \vv in {\vi,...,6}{\Edge(\v)(\vv)};
     };
     \SetVertexLabel
    \Vertices{circle}{0,1,2,3,4,5} 
\end{tikzpicture}
\caption{The transitive tournament $T_5$.}
\label{fig:transitive tournament}
\end{center}
\end{figure}

\begin{corollary}
    Transitive tournaments are (regularly) diagonal.
\end{corollary}

\begin{proof}
    Transitive tournaments are constructed by taking iterated cones, hence the statement follows from Proposition~\ref{prop:cones}.
\end{proof}

 The definition of transitive tournaments  makes it straightforward to exhibit the ranks of the (eulerian) magnitude homology groups.

\begin{remark}\label{rem:turnaments}
    We start with computing the (eulerian) magnitude homology groups of $T_n$ in bidegree~$(2,2)$. Although there are no cycles of length $3$, for each pair of vertices $i,j$ with $i < j$, there are $j-i-1$ paths of length $2$ from $i$ to $j$, meaning $\mathrm{rk}\,\EMH_{2,2}(i,j)=j-i-1$. 
    Here $\EMH_{2,2}(i,j)$ indicates the subgroup of $\EMH_{2,2}(G)$ generated by tuples starting at $i$ and ending at $j$. Therefore, accounting for all possible pairs $(i,j)$ with $i<j$, for $n\geq 2$ we have
    \[
    \mathrm{rk}\;\EMH_{2,2}(T_n) = \sum_{i=0}^{n-2}  \sum_{j=2}^n\mathrm{rk}(\EMH_{2,2}(i,j))=\sum_{i=0}^{n-2}  \sum_{j=i+2}^n (j-i-1)= \sum_{k=0}^{n}k(n-k) \ .
    \]
    Equivalently, we can also write the sum as follows:
    \[
    \EMH_{2,2}(T_n) = \sum_{i=1}^{n-1}  \mathrm{rk}\;\EMH_{1,1}(T_{i}) \ ,
    \]
    and a similar recursive formula generally holds in higher degrees. 
    For $n> k\geq 2$ we have the general recursive formula:
    \begin{equation}\label{eq:ranktrturn}
    \EMH_{k,k}(T_n) = \sum_{i=k-1}^{n-1}  \mathrm{rk}\;\EMH_{k-1,k-1}(T_{i}) \ ,
    \end{equation}
    based on the fact that to compute the number of ordered increasing sequences of $k$-tuples it is enough to fix the last (or initial) digits and then to compute the number of $(k-1)$-tuples on the resulting subsequences. Putting together, we have that the rank of $\EMH_{k,k}(T_n) $ equals the sum:
    \[
    \sum_{i_1=k-1}^{n-1} \sum_{i_2=k-2}^{i_1-1} \cdots \sum_{i_{k-1}=1}^{i_{k-2}-1} \mathrm{rk}\;\EMH_{1,1}(T_{i_{k-1}}) =  \sum_{i_1=k-1}^{n-1} \sum_{i_2=k-2}^{i_1-1} \cdots \sum_{i_{k-1}=1}^{i_{k-2}-1} \frac{i_{k-1}(i_{k-1}+1)}{2}\ ,
    \]
    where in the last  equality we used that the rank in bidegree $(1,1)$ is the number of directed edges. Observe that this quantity yields the number of $k$-simplices in the standard simplex $\Delta^n$. Indeed, we shall see that the homology of the complex $\EMH_{*,*}(T_n)$ is precisely the homology of the standard simplex $\Delta^n$. 
 \end{remark}

More generally, we can extend the result of Proposition~\ref{prop:cones} to the directed joins. We define the directed join of the directed graphs $G$ and $H$ to be the digraph $G*H$ on the vertices $V(G)\sqcup V(H)$ with edges $$E(G)\sqcup E(H)\sqcup \{(x,y)\mid x\in V(G), y\in V(H)\}\ .$$
Then, Proposition~\ref{prop:cones} generalizes as follows:

\begin{theorem}\label{thm:joins}
For  $G*H$ directed join of $G$ and $H$, the following formula holds:
\[
\EMH_{k,l}(G*H) \cong \bigoplus \EMH_{k_1,l_1}(G)\oplus\EMH_{k_2,l_2}(H) \ ,
\]
where the sum is taken over all $k_1+k_2+1=k$ and $l_1+l_2+1=l$.
\end{theorem}

\begin{proof}
First,  observe that the eulerian magnitude chains of $G*H$ consist of tuples $(x_0,\dots,x_{k_1},y_0,\dots, y_{k_2})$, with $(x_0,\dots,x_{k_1})\in \EMC_{k_1,l_1}$ and $(y_0,\dots,y_{k_2})\in \EMC_{k_2,l_2}$. As each vertex of $G$ is connected to each vertex of $H$ by a directed edge, the length of $(x_0,\dots,x_{k_1},y_0,\dots, y_{k_2})$, is $l=l_1+l_2+1$. 
Hence, as modules, we have a decomposition of $\EMC_{k,l}(G*H)$ as direct sum $\EMC_{k_1,l_1}(G)\oplus \EMC_{k_2,l_2}(H)$, via the map 
\[
(x_0,\dots,x_{k_1},y_0,\dots, y_{k_2})\mapsto ((x_0,\dots,x_{k_1}),(y_0,\dots, y_{k_2})) \ .
\]
Observe now that the first face maps, for $i=0,\dots,k_1-2$, act on the addendum $\EMC_{*,l_1}(G)$  (leaving the component $\EMC_{k_2,l_2}(H)$ unchanged) and the action agrees with the action by the face maps on $\EMC_{k_1,l_1}(G)$. Likewise, for $i=k_1+1,\dots,k_1+k_2+1$, the face maps act on the addendum $\EMC_{*,l_2}(H)$ leaving the first components unchanged. For $i=k_1-1,k_1$, the face maps delete either $x_{k-1}$ or  $y_0$ from the tuple. In such cases, the length of the tuple drops by one, hence the face maps $\partial_{k_1-1},\partial_{k_1}$ act trivially. As a consequence, we have a splitting of the differential on $\EMC_{k,l}(G*H)$ and an isomorphism 
\[
\EMC_{k,l}(G*H) \cong \bigoplus \EMC_{k_1,l_1}(G)\oplus\EMC_{k_2,l_2}(H) 
\]
of chain complexes. 
\end{proof}

Recall from Notation~\ref{notation:digraphs} that for an undirected graph~$G$, $\rho(G)$ denotes the directed graph obtained by replacing each undirected edge $\{v,w\}$ with both directed edges $(v,w)$ and $(w,v)$. 

\begin{corollary}\label{cor:joinscomplete}
    Directed joins of complete (undirected) graphs are regularly diagonal. Furthermore, we have: 
    \[
    \mathrm{rk}\, \EMH_{k,k}(\rho(K_n) * \rho(K_m))=\sum_{\substack{k_1,k_2  \\ k_1 + k_2 =k}}\left(\frac{n!}{(n-(k_1+1))!} + \frac{m!}{(m-(k_2+1))!}\right) \ .
    \]
\end{corollary}

We do not know of a characterization of regularly diagonal digraphs similar to Theorem~\ref{thm:charregdiag}. In fact, not all bipartite graphs are isomorphic to joins of complete graphs, but they always yield regularly diagonal digraphs. Hence, the naive characterization as joins of complete graphs of Corollary~\ref{cor:joinscomplete} does not hold.

\begin{question}
    Is there a complete characterization of regularly diagonal directed graphs?
\end{question}

\section{The complex of injective words on a digraph}
\label{sec:complex inj word digraph}

In this section, we first recall the complex of injective words on a set, and extend this notion to directed graphs. 

Let $W$ be a finite set. 

\begin{definition}\label{def:complinjwords}
    The complex of injective words $\Delta(W)$ on the set $W$ is the simplicial complex on the   injective   words $w = v_1 v_2 \cdots v_k$ of $W$  with alphabet the elements of~$W$.  The $p$-simplices are the injective words with $p+1$ letters.
\end{definition}

Saying otherwise, the simplicial complex $\Delta(W)$ is the simplicial complex on all  ordered sequences of distinct elements of $W$.
For $n\geq 1$, we let $D(n)\coloneqq \sum_{j=0}^n(-1)^j \frac{n!}{j!}$ denote the number of derangements (that is, the fixed point free permutations)  in the symmetric group $\Sigma_n$. The following theorem shows that the homotopy type of the complex of injective words is well-understood:

\begin{theorem}[\cite{zbMATH03647758,zbMATH03811580}]\label{thm:homotopycomplexinj}
    Let $W$ be a set of cardinality $|W|=n> 1$. Then, the complex of injective words $\Delta(W)$ is homotopy equivalent to 
    \[
    \Delta(W)\simeq \bigvee_{D(n)} S^{n-1} \ ,
    \]
    a wedge of spheres.
\end{theorem}

We now extend the notion of complexes of injective words to directed graphs. 

\begin{definition}\label{def:injwordsgraph}
    For a directed graph $G$ with vertex set~$V(G)$, we define the complex of injective words on $G$ to be the simplicial complex $\Inj(G)$ generated by all injective words $w=x_0\cdots x_k$ on the alphabet $V(G)$, with the requirement that for all $i<j$ there is a directed path in $G$ from $x_i$ to $x_j$.  
\end{definition}

In Definition~\ref{def:injwordsgraph} it is required that there exists a path from $x_i$ to $x_j$ for all $i<j$. This allows $\Inj(G)$ to be a simplicial complex. The next example shows that we can not drop this assumption. 

\begin{example}
    Let $L_3$ be the linear directed graph on three vertices $\{x_0,x_1,x_2\}$, with edges $(x_0,x_1),(x_1,x_2)$. Then, the associated complex of injective words $\Inj(L_3)\cong \Delta^2$ is the standard 2-simplex. Note that  $(x_0,x_2)$ is not an edge of $L_3$.  
\end{example}

\begin{example}\label{ex:bicomplete}
If $BK_n=\rho(K_n)$ is the complete digraph with both reciprocal
edges between any pair of vertices, then $\Inj(BK_n)=\Delta(V)$,  the classical complex of injective words on the set $V$ of vertices of $BK_n$.    
\end{example}

In the general case, we shall see that $\Inj(G)$ is a directed flag complex, hence the complexity of its homotopy type can be arbitrary. To see it,  consider the transitive closure of $G$.  Its vertices  are the vertices of $G$ and there is a directed edge~$(x,y)$ if and only if there is a directed path from $x$ to $y$ -- including the trivial constant path. The transitive closure yields a preorder, since the graph~$G$ has no multiple directed edges. In the literature related to magnitude homology, this preorder is also called the \emph{reachability preorder}~$\Pre(G)$ of $G$  -- \textit{cf}.~\cite{zbMATH07868225, zbMATH07844814,arXiv:2312.01378}. We will also follow this notation. First, observe that passing to the reachability preorder does not affect the complex of injective words:

\begin{remark}\label{rem:GpreG}
    The complexes of injective words on $G$ and $\Pre(G)$ are isomorphic.
\end{remark}

Recall that an \emph{ordered $k$-clique} of a directed graph~$G$ is a totally ordered $k$-tuple $(v_1,...,v_k)$ of vertices of~$G$ with the property that, for every $i < j$, the pair $(v_i,v_j)$ is an ordered edge of~$G$. Then, for a given digraph $G$, we can construct the associated \emph{directed flag complex} $\mathrm{dFl}(G)$, that is the ordered simplicial complex on the vertices $V(G)$ of~$G$ whose $k$-simplices are all possible ordered $(k+1)$-cliques of~$G$.

\begin{proposition}\label{prop:inggraphs}
    Let $G$ be a digraph. Then, we have an isomorphism $\Inj(G)\cong\dFl(\Pre(G))$. 
\end{proposition}

\begin{proof}
Consider the  map 
\[
\Inj(G)\to\dFl(\Pre(G))
\]
sending an injective word $x_0\cdots x_k$ to the ordered clique $(x_0,\dots,x_k)$. The map is well-defined because, by definition, for all $i<j$ the distinct vertices $x_i$ and $x_{j}$ are connected by a directed path in $G$ (hence, by an edge in $\Pre(G)$), and they yield an ordered clique in $\Pre(G)$. Both the differentials on  $\Inj(G)$ and $\dFl(\Pre(G))$ act by forgetting the $i$-th entry, hence the map is simplicial. Furthermore, it is an isomorphism of simplicial complexes with inverse the map $\dFl(\Pre(G))\to\Inj(G)$ sending the ordered clique $(x_0,\dots,x_k)$ to the injective word $x_0\cdots x_k$.
\end{proof}

Let $H$ be an undirected graph on $n$ vertices, and $K_n$ be the complete graph on $n$ vertices. Consider $G=\rho(H)$.  Then, the reachability preorder $\Pre(G)$ is isomorphic to the complete digraph $\rho(K_n)$. Consequently, the associated  directed flag complex $\dFl(\Pre(G))$ is isomorphic to $ \dFl(\rho(K_n))=\Delta(V(G))$,  the standard complex of injective words.

\begin{remark}\label{rem:dFl_posets}
    If  $G$ is the directed graph associated to a poset $P$, that is the graph on the same vertices of $P$ with an edge $(x,y)$ if and only if $x\leq y$ in $P$,  then $\dFl(\Pre(G))$ is the standard order complex of $P$\footnote{Recall that the order complex of a poset $(P,\leq)$ is the simplicial complex on the vertex set $P$ whose $k$-simplices are the chains $x_0 < \cdots < x_k$ of $P$.}. 
\end{remark}

By Theorem~\ref{thm:homotopycomplexinj}, the complex of injective words is, up to homotopy, always a wedge of spheres. This is not the case for the complex of injective words on digraphs, which can have any homotopy type. To see it, first observe that as $\Pre(P)\cong P$, we have an isomorphism $\Inj(P)\cong \dFl(P)$ -- that is the order complex of $P$.
Given a simplicial complex $\Sigma$,  consider its face poset $P=F(\Sigma)$, whose nerve is  (the barycentric subdivision of) $\Sigma$. Then, the complex of injective words  $\Inj(F(\Sigma))$ is homotopy equivalent to~$\Sigma$; which was arbitrary and can have  any homotopy type.

	\begin{example}\label{ex:spheres}
	    Consider the following digraphs with reciprocal edges:
	
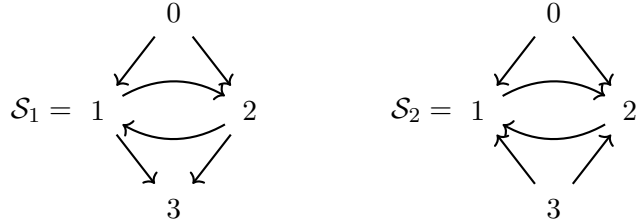
\begin{center}
		\begin{tikzpicture}
			\tikzset{vertex/.style = {shape=circle, draw=black, minimum size=1.5em}}
			\tikzstyle{arc}=[shorten >= 1pt,black, shorten <= 1pt,->,thick]
		  
			\node[left] at (-1.2,0) {\(\mathcal{S}_1= \, \)};
			\node[vertex] (0) at  (0,1.3) {0};
			\node[vertex] (1) at  (-1,0) {1};
			\node[vertex] (2) at  (1,0) {2};
			\node[vertex] (3) at  (0,-1.3) {3};
			\draw[arc] (0) to (1);
			\draw[arc] (0) to (2);
			\draw[arc] (2) to[bend left] (1);
			\draw[arc] (1) to[bend left] (2);
			\draw[arc] (2) to (3);
			\draw[arc] (1) to (3);
			
			\node[left] at (3.8,0) {\(\mathcal{S}_2=\, \)};
			
			\node[vertex] (0) at  (5,1.3) {0};
			\node[vertex] (1) at  (4,0) {1};
			\node[vertex] (2) at  (6,0) {2};
			\node[vertex] (3) at  (5,-1.3) {3};
			\draw[arc] (0) to (1);
			\draw[arc] (0) to (2);
			\draw[arc] (2) to[bend left] (1);
			\draw[arc] (1) to[bend left] (2);
			\draw[arc] (3) to (1);
			\draw[arc] (3) to (2);
		\end{tikzpicture}
	\end{center}
    Note that the associated complexes of injective words are isomorphic (and homeomorphic to the 2-sphere). By Theorem~\ref{thm:joins}, the graph $\mathcal{S}_2$ is regularly diagonal. Also the graph $\mathcal{S}_1$ is regularly diagonal. However, the associated eulerian magnitude homology groups are not isomorphic, such as in bidegree~$(3,3)$. 
	\end{example} 

The last example shows that, although the complexes of injective words on two graphs are homotopy equivalent, we can have different eulerian magnitude homology groups. In the next section we aim to investigate this property more accurately.

\section{The injective nerve of categories and the regular magnitude-path spectral sequence}
\label{sec:regular MPSS}

The goal of this section is to construct the regular magnitude-path spectral sequence, in analogy with the magnitude-path spectral sequence developed in~\cite{zbMATH07731261}. To do so, it shall be more convenient to use the categorical language, and to generalize the framework of eulerian magnitude homology to quivers. In the first subsection we introduce the tool of \emph{injective nerves} of categories, and then we shall apply it to the reachable categories, getting the regular magnitude-path spectral sequence. As an application of the construction, we shall provide some computations of regular path homology of regularly diagonal graphs.

\subsection{The injective nerve}\label{sec:injnerve}
Let \(\mathbf{2}\) denote the category consisting of the  objects~\(E\) and $V$, with two non-identity morphisms \(s,t \colon E \to V\) called the \emph{source} and the \emph{target}.   Let \(\mathbf{Fin}\) be the full subcategory of $\mathbf{Set}$ of finite sets. By a (finite) \emph{quiver} we shall mean a functor \(Q \colon \mathbf{2} \to \mathbf{Fin}\). We shall also represent, and identify, a quiver~$Q$ by the corresponding directed graph with set of vertices $V$, and set of edges $E$. Morphisms of quivers are natural transformations of functors. The category~$\mathbf{Quiver}$ of finite quivers and morphisms of quivers is the functor category $\mathbf{Fun}(\mathbf{2},\mathbf{Fin})$.

There is an adjunction
 \begin{equation}\label{eq:adjunction}
 \begin{tikzcd}
\mathbf{Quiver}\arrow[r,bend left, "\mathrm{Free}"]  & \Cat \arrow[l, bend left, "U"'] 
 \end{tikzcd}    
 \end{equation}
 between the category of (finite) quivers and the 1-category $\Cat$ of small (finite) categories and functors. In this adjunction, the functor $\mathrm{Free}$ is the free functor that turns a quiver into a category, and $U$ is the corresponding forgetful functor. 

 Recall from Section~\ref{sec:complex inj word digraph} that an \emph{ordered $k$-clique} of a directed graph $G$ is a totally ordered $k$-tuple $(v_1,...,v_k)$ of vertices of~$G$ with the property that, for every $i < j$, the pair $(v_i,v_j)$ is an ordered edge of~$G$. This definition extends to quivers in the natural way and, for any given quiver $Q$, we can construct the associated {directed flag complex} $\mathrm{dFl}(Q)$ as well. Recall also that the directed flag complex of a quiver is an \emph{ordered simplicial complex}, and its simplices are determined by the vertices. 
 
 Let  $\Delta\mathbf{Set}$ be the category of semi-simplicial sets, and  $\mathbf{sSet}$ the category of simplicial sets. 
Recall that a semi-simplicial set is a sequence $\{X_n\}_{n\in\mathbb{N}}$ of sets together with functions called \emph{face maps} between them which encode that an element in $X_{n+1}$ has $n+1$ faces (\textit{i.e.}~boundary segments), and that such faces   are elements of $X_n$; similarly, a simplicial set is a collection of sets $\{S_n\}_{n \in \mathbb{N}}$ together with \emph{face maps} $f_i\colon S_{i+1}\to S_i$ and \emph{degeneracy maps} $s_i\colon S_i \to S_{i+1}$, where the $f_i$'s are constructed similarly to the semi-simplicial case, and the $s_i$'s encode which $(i+1)$-simplices are really just $n$-simplices regarded as degenerate $(n+1)$-simplices.

 \begin{remark}\label{rem:tosimplsets}
 An ordered simplicial complex $K$ yields a semi-simplicial set, hence  a simplicial set. This is given by adding to  $K$ all possible degenerate simplices. More precisely, the $n$-simplices of the simplicial set $\iota(K)$ associated to the order simplicial complex $K$ consist of tuples $[v_{i_0},\dots, v_{i_n}]$ with $v_{i_j}\leq v_{i_{j+1}}$ and such that the set of vertices $\{v_{i_0},\dots, v_{i_n}\}$ spans a simplex of $K$  -- see, \emph{e.g.}~\cite[Example~3.3]{zbMATH06035442}. Therefore, for every simplex    $[v_{i_0},\dots, v_{i_m}]$ of $K$ we have in $\iota(K)$ all the (degenerate) simplices of the form 
 \[
 [v_{i_0},\dots,v_{i_0},\dots, v_{i_1},\dots, v_{i_1},\dots, v_{i_m},\dots,  v_{i_m}] \ ,
\]
for any number of repetitions of each of the vertices. The non-degenerate
simplices of $\iota(K)$ are in bijection with the simplices of $K$.
\end{remark}

In view of Remark~\ref{rem:tosimplsets}, for a small category $\bC$, we can consider the following simplicial set:
\begin{equation}\label{eq:injnerve}
N^{\iota}(\bC)\coloneqq \iota(\dFl(U(\bC))) \ ,
\end{equation}
where $U\colon \Cat\to \Quiver$ is the forgetful functor.

\begin{definition}\label{def:injnerve}
    For a small category $\bC$, the simplicial set~$N^{\iota}(\bC)$ is called the \emph{injective nerve} of $\bC$. 
\end{definition}

We can characterize the non-degenerate simplices of the injective nerve as the functors that are injective on objects. To see it, first consider the category $[n]$, that is the category with objects the natural numbers $0,\dots, n$ and with morphisms $i\to j$ if and only if $i\leq j$. Saying otherwise, $[n]$ is the free category generated by the transitive tournament~$T_n$. Denote by $\Cat^{\iota}(\mathbf{D},\bC)$ the family of functors from $\mathbf{D}$ to~$\bC$ which are injective on objects. 

\begin{lemma}\label{lemma:setnon-deg}
 The set of non-degenerate $n$-simplices of $N^{\iota}(\bC) $ is  $\Cat^{\iota}([n],\bC)$.
\end{lemma}

\begin{proof}
Recall the adjunction of Eq.~\eqref{eq:adjunction}. The set of ordered $(n+1)$-cliques of the quiver $U(\bC)$ is given by the ordered tuples $(c_0,\dots,c_n)$ consisting of objects of $\bC$, with the further requirement that for all $i<j$ there is a morphism $f_{i,j}\colon c_i\to c_j$ in $\bC$. Then, each such tuple defines a functor $S\colon [n]\to \bC$ with $S(i)\coloneqq c_i$ and $S(i<j)\coloneqq f_{i,j}$. The functor $S$ is injective on objects because the set of vertices is ordered (and without repetitions). Vice versa, every functor $S\colon[n]\to \bC$ which is injective on objects yields an ordered $(n+1)$-clique, hence an $n$-simplex of the directed flag complex -- and, consequently, a non-degenerate simplex of $N^{\iota}(\bC) $.
\end{proof}
    
We now shall  justify the term \emph{injective nerve} used in Definition~\ref{def:injnerve}. Let $\Delta$ be the  simplex category, and denote by $\Delta_+$ the wide subcategory of $\Delta$ containing only the injective functions. The category $\Delta_+$ is (equivalent to) the category of finite totally ordered sets and order-preserving injections.  Recall that for a given (small) category~$\bC$, the nerve $N(\bC)$ of $\bC$ is the simplicial set
\[
N(\bC)\colon \Delta^{\op}\hookrightarrow\Cat^{\op}\xrightarrow{\Cat(-,\bC)} \Set \ ,
\]
where $\Cat(-,\bC)$ denotes the Hom-functor. If we restrict to the category $\Delta_+$, we get the nerve of $\bC$ as a semi-simplicial set -- which means, without taking degeneracies. We shall further restrict it to the set of functors $\Cat^{\iota}(-,\bC)$ which are injective on objects; hence, we  consider the composition 
\[
\Delta^{\op}_+\hookrightarrow\Cat^{\op}\xrightarrow{\Cat^{\iota}(-,\bC)} \Set \ .
\]
The following result justifies  the term of injective nerve used so far;
        
\begin{proposition}\label{prop:injnerve}
    The injective nerve $N^\iota(\bC)$ of a small category $\bC$ is the simplicial set associated to 
    \[
    \Delta^{\op}_+\hookrightarrow\Cat^{\op}\xrightarrow{\Cat^{\iota}(-,\bC)} \Set \ .
    \]
\end{proposition}

\begin{proof}
    By Lemma~\ref{lemma:setnon-deg}, the set of non-degenerate simplices of the injective nerve of a small category~$\bC$ can be described as the set of functors $[n]\to\bC$ which are injective on objects. The $i$-th face map operator in $\Delta_+$ acts by forgetting the $i$-th entry in the tuples $(c_0,\dots,c_k)$. Hence, the face maps are coherent with the face maps of the directed flag complex construction, yielding an isomorphism of semi-simplicial sets.  The statement follows after considering the associated simplicial sets. 
\end{proof}

As a consequence, Proposition~\ref{prop:injnerve} can be taken as alternative (categorical) definition of the injective  nerve; in the following, we shall use both viewpoints.

\begin{remark}
 Recall that the nerve of a category is a fully faithful functor $N\colon\Cat\to\mathbf{sSet}$. In view of Proposition~\ref{prop:injnerve}, also the injective nerve $N^\iota$ is a faithful functor. However,  it is generally not  full.
\end{remark}

One may ask whether the injective nerve of a category is a (simplicial) nerve, at the cost of changing the base category. However, the following result tells us that in general the injective nerve is not the nerve of any category -- nor of any quasicategory.

\begin{proposition} 
The inner horns of the injective nerve $N^{\iota}(\bC)$ do not generally have fillers. Hence, $N^{\iota}(\bC)$ is not the nerve of any (quasi)category.
\end{proposition}

\begin{proof}
    Consider the category with two objects $0,1$, and non-trivial morphisms $0\to 1$, $1\to 0$. Consider the inner horn $\lambda^n_i$ in the case $n=2$, $i=1$ given by the arrows $0\to 1$, $1\to 0$. Then, $\lambda^2_1$ has no lift to $\Delta^2$ -- because the simplex $(0,1,0)$ does not belong to $N^{\iota}(\bC)$.
\end{proof}

Observe that the inclusion $\Cat^{\iota}(-,\bC)\subseteq \Cat(-,\bC)$ induces a natural transformation
\[
N^\iota(\bC)\to N(\bC) \ .
\]
 It would be interesting to know 
when $N^\iota(\bC)$ and $N(\bC)$ are homotopy equivalent. In complete analogy with Remark~\ref{rem:dFl_posets},  we have the following partial result.
Recall that each poset can be seen as a category in the standard way. Let $\mathbf{Poset}$ be the category of (finite) posets, and poset maps. There is an inclusion functor $\mathbf{Poset}\hookrightarrow \Cat$. The (injective) nerve  of a poset $P$ is then the (injective)  nerve of $P$, seen as a category. 

\begin{proposition}\label{prop:posets}
      Let $P$ be a poset. Then,
    the inclusion $\Cat^{\iota}(-,P)\subseteq \Cat(-,P)$ induces an isomorphism
$
N^\iota(P)\to N(P)
$
of simplicial sets.
\end{proposition}

\begin{proof}
The degenerate $n$-simplices of the nerve~$N(P)$ consist of tuples of composable arrows in $P$, at least one of whose components
is an identity morphism; the non-degenerate $n$-simplices  are in bijection with the ordered tuples $[x_0,\dots,x_n]$ of elements of~$P$, with $x_i< x_{i+1}$. Since~$P$ is a poset, it does not contain loops, hence all objects appearing in a non-degenerate simplex~$[x_0,\dots,x_n]$ are in fact distinct; \emph{i.e.}~$x_i\neq x_j$ for $i\neq j$. As a consequence, the set of non-degenerate $n$-simplices of the nerve~$N(P)$ is in bijection with the set of functors  $[n]\to P$ which are injective on objects. In view of Lemma~\ref{lemma:setnon-deg}, this means that  the inclusion $\Cat^{\iota}(-,P)\subseteq \Cat(-,P)$  induces a bijection between the non-degenerate simplices of $N(P)$ and of $N^\iota(P)$. As the face maps and degeneracies are defined in a coherent way, this is enough to prove that $N^\iota(P)\to N(P)$ is an isomorphism of simplicial sets.  
\end{proof}

\subsection{The regular magnitude-path spectral sequence}

With en eye towards the regular magnitude-path spectral sequence, we are interested in the injective nerve of reachability categories, whose definition we now recall. 

Let $Q$ be a directed graph or, more generally, a quiver.

\begin{definition}\label{def:reach_G}
The  \emph{reachability category} \(\Reach_{Q}\)  is the category with objects the vertices of~\(Q\), and for $v,w\in Q$, the Hom-set \(\Reach_{Q}(v,w)\) is defined as follows:
\[
\Reach_{Q}(v,w)\coloneqq 
\begin{cases}
* & \text{if there is a path from $v$ to $w$ in $Q$};  \\
\emptyset & \text{ otherwise}.\
\end{cases}
\]
The Hom-set \(\Reach_{Q}(v,v)\) is defined as the identity at \(v\).
\end{definition}

    The reachability category is an EI-category, that is, all endomorphisms are isomorphisms.

\begin{example}
 If $Q$ is a directed tree, then $\mathrm{Free(Q)}$ and $\Reach_Q$ are isomorphic categories. More generally, $\mathrm{Free(Q)}$ and $\Reach_Q$ are isomorphic if and only if the quiver~$Q$ does not contain directed cycles nor quasi-bigons -- see~\cite[Proposition~3.8]{zbMATH07844814}.   
\end{example}

 Using the description of the injective nerve of Remark~\ref{rem:tosimplsets}, we can concretely describe the simplicial set $N^{\iota}(\Reach_Q)$ as the simplicial set associated to the directed flag complex $\dFl(U(\Reach_Q))$. The directed flag complex of $U(\Reach_Q)$ has the vertices of $Q$ as $0$-simplices, and as $p$-simplices the (ordered) tuples $(x_0,\dots,x_p)$ of distinct vertices of $Q$ with the property that for each $i<j$ there is a directed path in $Q$ from $x_i$ to $x_j$. By transitivity, this means that the non-degenerate simplices of the injective nerve~$N^\iota(\Reach_Q)$ are given by tuples $(x_0,\dots,x_p)$ of distinct vertices of $Q$ with the property that for each $i=0,\dots,p-1$ there is a directed path in $Q$ from $x_i$ to $x_{i+1}$.

\begin{remark}\label{rem:reach-inj}
    Let $G$ be a directed graph. Then, from the definitions of injective nerve of a category (Remark~\ref{rem:tosimplsets} and Definition~\ref{def:injnerve}) and of reachability category (Definition~\ref{def:reach_G}) we get an isomorphism  $U(\Reach_G)\cong \Pre(G)$. As a consequence, by Proposition~\ref{prop:inggraphs}, we have
    \[
N^\iota(\Reach_G)=\iota(\dFl(U(\Reach_G))\cong \iota(\dFl(\Pre(G))\cong \iota(\Inj(G)) \ ,
    \]
    and the injective nerve of $\Reach_G$ is isomorphic to the simplicial set associated to the complex of injective words on $G$.
\end{remark}

Recall that, for a (ordered) simplicial complex $K$, the geometric realization $|\iota(K)|$ of the associated simplicial set returns the geometric realization of $K$. 
    
\begin{example}
      If $\bC=[n]$, then $|N^{\iota}([n])|\cong| N([n])|$ is contractible (and  homeomorphic to the standard simplex~$\Delta^n$).  More generally, if $Q$ is any connected subgraph of $U([n])$, then its reachability $\Reach_Q$ is isomorphic to $[n]$, hence its (injective) nerve is also contractible. This case recovers the computation on transitive tournaments developed in Remark~\ref{rem:turnaments}.
\end{example}

We provide also the following example, which shows that the homotopy type of the injective nerve and of the nerve of a reachability category can be quite different:

\begin{example}
    Let $BK_n=\rho(K_n)$ be the bicomplete digraph on $n$ vertices. In view of Example~\ref{ex:bicomplete}, Remark~\ref{rem:GpreG}, and Proposition~\ref{prop:inggraphs}, the injective nerve $N^\iota(\Reach_{BK_n})$ is isomorphic to the (simplicial set associated to the) classical complex of injective words. More generally, this holds if $G$ is any connected subgraph of~$BK_n$, which is also strongly connected. 
    Hence, $|N^{\iota}(\Reach_G)|\simeq| \Delta(V(BK_n))|$ is homotopy equivalent to a wedge of spheres of dimension~$n-1$ by Theorem~\ref{thm:homotopycomplexinj}. 
    On the other hand, the  nerve $N(\Reach_G)$ is contractible. 
\end{example}  
   
   In parallel with Remark~\ref{rem:Kunneth}, we can not expect the injective nerve of reachability categories to preserve products, not even up to homotopy:
   
   \begin{example}\label{ex:prod}
       Consider the (Cartesian) product $\rho(K_2)\times \rho(K_2)$. Then, the injective nerve of the associated reachability category is isomorphic to the (simplicial set associated to the) complex of injective words on four letters. Hence, up to homotopy, it is a wedge of $3$-spheres. On the other hand, we have $|N^\iota(\Reach_{\rho(K_2)})|\simeq S^1$.
   \end{example}
   
   For similar reasons as described in Example~\ref{ex:prod}, we can not expect excision to generally hold, not even when the morphisms between reachability categories are Dwyer morphisms.

We now proceed with the construction of the regular magnitude-path spectral sequence.
Let $Q$ be a quiver, and recall from Eq.~\eqref{eq:len} the notion of the length of a tuple. 
For each $\ell\in \mathbb{N}$, there are subsimplicial sets $F_\ell N^\iota(\Reach_Q)$ and $F_\ell N(\Reach_Q)$ of $N^\iota(\Reach_Q)$ and $N(\Reach_Q)$, respectively, defined as follows. The $k$-simplices of the simplicial set $F_\ell N(\Reach_Q)$ are the $k$-simplices $(x_0,\dots,x_k)$ of $N(\Reach_Q)$ of length $\len(x_0,\dots,x_k)\leq \ell$. Analogously, the $k$-simplices of $F_\ell N^\iota(\Reach_Q)$ are the $k$-simplices $(x_0,\dots,x_k)$ of $N^\iota(\Reach_Q)$ of length $\len(x_0,\dots,x_k)\leq \ell$. We have a commutative diagram
\begin{center}
\begin{tikzcd}
F_\ell N^\iota(\Reach_Q)\arrow[d] \arrow[r] & F_\ell N(\Reach_Q)\arrow[d]\\
N^\iota(\Reach_Q) \arrow[r]&N(\Reach_Q)
\end{tikzcd}    
\end{center} 
where all maps are inclusions of simplicial sets. Extending the definition to all $\ell\in\mathbb{N}$, we get the resulting filtrations $F_* N^\iota(\Reach_Q)\subseteq N^\iota(\Reach_Q)$ and 
$F_* N(\Reach_Q)\subseteq N(\Reach_Q)$ of simplicial sets of the injective nerve, and of the nerve of the reachability category, respectively. The following is a direct consequence of the definitions:

\begin{remark}
    The inclusions $F_\ell N^\iota(\Reach_Q)\subseteq F_\ell N(\Reach_Q)$ extend to a map of filtrations.  
\end{remark}

\sloppy
Consider the spectral sequences associated to the filtration  
$F_* N(\Reach_Q)\subseteq N(\Reach_Q)$; this is called the \emph{magnitude-path spectral sequence}~\cite{zbMATH07731261} (see also~\cite{zbMATH07868225} for the construction of the magnitude-path spectral sequence in the context of filtered simplicial sets). Likewise, we consider the spectral sequence associated to the filtration  $F_* N^\iota(\Reach_Q)\subseteq N^\iota(\Reach_Q)$ of the injective nerve, which we call the \emph{regular magnitude-path spectral sequence}. In the next result, we follow the notation of the pages of the spectral sequences as in~\cite{zbMATH07731261}.

\begin{theorem}\label{thm:rmpss}
Let $G$ be a directed graph.
    The regular magnitude-path spectral sequence $IE_*^{*,*}$ associated to the filtration $F_* N^\iota(\Reach_G)\subseteq N^\iota(\Reach_G)$ satisfies the following properties:
    \begin{enumerate}
        \item The first page $IE_*^{*,1}$ is the eulerian magnitude homology of $G$.
        \item The diagonal of the second page $IE_*^{*,2}$ is the
        regular path homology.
        \item It converges to the  homology of the complex of injective words on $G$.
        \item The inclusions $F_\ell N^\iota(\Reach_G)\subseteq F_\ell N(\Reach_G)$ extend to a map of spectral sequences from $IE_*^{*,r}$ to the magnitude-path spectral sequence.
    \end{enumerate}
\end{theorem}

\begin{proof}
Consider the filtration $F_* N^\iota(\Reach_G)$ of $ N^\iota(\Reach_G)$. For a simplicial set~$S$, consider its normalized chain complex $\mathrm{N}(S)$. If $S$ is filtered, then also the normalized chain complex of $S$ is filtered, with filtration given by $F_l\mathrm{N}(S)=\mathrm{Im}(\mathrm{N}(F_l S)\to \mathrm{N}(S))$. The $0$-page of the spectral sequence $IE^{*,*}_*$ associated to this filtration is given by the quotient
\[
IE_*^{\ell,0}= \frac{F_\ell \mathrm{N}(N^\iota(\Reach_G))}{F_{\ell-1}\mathrm{N}(N^\iota(\Reach_G))} 
\]
of chain complexes -- see also~\cite{zbMATH07868225}. By construction, $IE_*^{\ell,0}$ is isomorphic to  the eulerian chain complex~$\EMC_{*,\ell}(G)$. The induced differential $\partial\colon IE^{\ell,0}_k\to IE^{\ell,0}_{k-1}$ coincides with the differential of the eulerian magnitude chains, hence the isomorphism
\[
IE^{\ell,1}_k\cong \EMH_{k,\ell}(G)
\]
follows. When restricted to $k=\ell$,  the induced second differential $\tilde\partial\colon IE^{\ell,1}_k\to IE^{\ell-1,1}_{k-1}$ computes the 
regular path homology groups of $G$ by~\cite[Proposition~6.11]{zbMATH07731261} (adapted to the strongly regular elementary paths); that is, the homology groups of the  chain complex $(IE^{k-*,1}_{k-*},\tilde\partial)$ are the 
regular path homology groups of $G$, proving the second item in the statement. 

By construction,  the considered spectral sequence converges to the  homology of the injective nerve~$N^\iota(\Reach_G)$. In view of  Remark~\ref{rem:reach-inj}, the  homology of $N^\iota(\Reach_G)$ is the homology of 
the complex of injective words on $G$. 

To conclude, every morphism of filtrations yields a morphism of spectral sequences, which shows that the inclusions $F_\ell N^\iota(\Reach_G)\subseteq F_\ell N(\Reach_G)$ extend to a map of spectral sequences from the magnitude-path spectral sequence to $IE_*^{*,r}$.
\end{proof}

We now show some applications. 
First, observe that we can apply Theorem~\ref{thm:rmpss} to undirected graphs as well -- by replacing an undirected  graph $G$ with $\rho(G)$. As a consequence of Proposition~\ref{prop:vanishingdiagonal} and of Theorem~\ref{thm:rmpss}, we have that all graphs without 3- and 4-cycles have trivial regular path homology in degree $k\geq 2$:

\begin{corollary}
    Let $G$ be an undirected  graph without 3- and 4-cycles. Then, its regular path homology is 0 in degree $k\geq 2$.
\end{corollary}

\begin{proof}
Recall that $\EMH_{k,\ell}(G)\cong\EMH_{k,\ell}(\rho(G))$.
   By Theorem~\ref{thm:rmpss}, the regular path homology of $G$ is the homology of the chain complex $(\EMH_{*,*}(G),\tilde\partial)$ appearing as diagonal in the second page of the regular magnitude-path spectral sequence. Then, the result follows by Proposition~\ref{prop:vanishingdiagonal}.
\end{proof}

As a main consequence of Theorem~\ref{thm:rmpss}, we  provide computations of regular path homology groups of  digraphs; in fact, the following holds:

\begin{proposition}\label{prop:gjaja}
    Let $G$ be a connected regularly diagonal  directed graph. 
    Then, the regular path homology groups of $G$ are the homology groups of $\Inj(G)$.
\end{proposition}

\begin{proof}
    The regular magnitude-path spectral sequence converges to the homology of the complex of injective words~$\Inj(G)\cong\dFl(\Pre(G))$.  If $G$ is regularly diagonal, then the first page $IE^1_{*,*}$ of the regular magnitude-path spectral sequence is non-zero only in bidegrees $(\ell,k)$ with $k=\ell$, and similarly for the second page. Hence,  the regular magnitude-path spectral sequence converges at the second page where all the non-trivial groups are concentrated on the diagonal. The result follows, as the diagonal coincides with the regular path homology of $G$. 
\end{proof}

\begin{corollary}
    Let $T$ be a finite directed tree, a bipartite graph with alternating orientation or a transitive tournament. Then, its regular path homology is trivial.
\end{corollary}

\begin{proof}
We have seen that finite directed tree,  bipartite graphs with alternating orientation or  transitive tournaments are regularly diagonal. Furthermore,  the  directed flag complexes associated to the respective reachability preorders are contractible. Hence, by Proposition~\ref{prop:gjaja}, their regular path homology groups are  trivial.
\end{proof}

Consider the directed graphs of Example~\ref{ex:spheres}. They are both regularly diagonal and connected, with the same number of vertices. Their regular path homology  is the homology of the $2$-sphere, as $\dFl\circ \Pre$ yields the $2$-sphere in both cases. However, observe that the associated filtrations are different, and such difference is taken into account by the associated eulerian magnitude homology groups.

\begin{corollary}\label{cor:complete}
    The regular path homology of the complete graph $K_n$ is concentrated in degree $n-1$ and is of rank the number of derangements $D(n)$.
\end{corollary}

\begin{proof}
    The graph $K_n$ is regularly diagonal. The regular magnitude-path spectral sequence converges to the homology of the 
    complex of injective words. By Theorem~\ref{thm:homotopycomplexinj} this is, up to homotopy, a wedge of spheres in degree $n-1$ on the number of derangements of $n$.
\end{proof}

  It was shown in~\cite{zbMATH06435208} that path homology is homotopy invariant. This was generalized in~\cite[Proposition~8.1]{zbMATH07731261} to the pages of the magnitude-path spectral sequence, for $r\geq 2$.
We now use Corollary~\ref{cor:complete} to show that regular path homology is not homotopy invariant (as opposed to path homology).

\begin{corollary}\label{cor:nonhinvpathhom}
    Regular path homology is not invariant with respect to homotopy equivalence of digraphs. 
\end{corollary}

\begin{proof}
Recall that complete graphs are regularly diagonal, hence the second page of the regular magnitude-path spectral sequence is non-trivial only on the diagonal $k=\ell$.  By \cite[Example~3.11]{zbMATH06435208} (see also \cite[Example~5.6]{zbMATH07731261}), the complete graph~$K_n$ is contractible and homotopy equivalent to the one-vertex graph. If the pages of the regular magnitude-path spectral sequence were homotopy invariant, then, by restriction, also regular path homology would be homotopy invariant. However, the regular path homology of the complete graph $K_n$ is concentrated in degree $n-1$ and is of rank the number of derangements $D(n)$, whereas the regular path homology of the one-point graph is trivial. 
\end{proof}

\begin{remark}\label{rem:homotopy}
    Chaplin, Harrington and Tillmann study in~\cite{chaplin2024notion} the homology of the directed flag complex and, by comparison with path homology, they establish a homotopy-like equivalence relation on digraph maps, under which equivalent maps induce the same morphisms on the homology of the directed flag complex.
    Though it may be possible to exploit their work to shed light on potential notions of homotopy invariance for the regular magnitude-path spectral sequence, pursuing this line of inquiry would require a substantial detour from the main objectives of the present paper, therefore we leave  this question open. 
\end{remark}

\subsection{Decategorification} We conclude with some considerations on the decategorification of  eulerian magnitude homology.

Recall that the magnitude $\#G(q)$ of a graph~$G$ is equivalently defined as follows -- see~\cite[Definition~2.1 \& Proposition~3.9]{leinster2019magnitude}:
\begin{equation}\label{eq:magnitudegraphs}
    \#G(q)= \sum_{k=0}^{\infty} (-1)^k \sum_{x_0\neq \dots\neq x_k} q^{d_G(x_0,x_1)+\dots+d_G(x_{k-1},x_k)}
\end{equation}
where $x_0, \dots, x_k$ run across all vertices of $G$. By \cite[Theorem 2.8]{hepworth2015categorifying}, this agrees with the decategorification of magnitude homology, that is
\begin{equation*}
    \sum_{k,\ell}(-1)^k\mathrm{rk}\;\MH_{k,\ell}(G) q^\ell = \#G(q)
\end{equation*}
for any undirected graph $G$.
In analogy with the definition of magnitude of a graph, we can define the \emph{regular magnitude} of a graph $G$ as the polynomial
\begin{equation}\label{eq:eumagnitudegraphs}
    \#_r G(q)= \sum_{k=0}^{|V(G)|} (-1)^k \sum_{x_0\neq \dots\neq x_k} q^{d_G(x_0,x_1)+\dots+d_G(x_{k-1},x_k)}
\end{equation}
where the vertices $x_0,\dots,x_k$ appearing in the sum are all distinct. The regular magnitude  $\#_r G(q)$ is not the magnitude of a graph, unless $G$ is a point. In fact, the first is a polynomial whereas the magnitude of a graph is a power series.  We give a different proof which uses the (regular) magnitude-path spectral sequences, and which might shed some light on the behavior of $\#_r G(q)$. 

\begin{proposition}\label{prop:eulmag}
Let $G$ be an undirected connected graph which is not a point. 
    Then, the decategorification of eulerian magnitude homology is  not the magnitude of any graph.
\end{proposition}

\begin{proof}
     By \cite[Theorem 2.8]{hepworth2015categorifying}, the magnitude of a graph~$G$ agrees with the decategorification of magnitude homology of $G$. The evaluation at $q=-1$ gives the sum
    \begin{equation}\label{eq:dec}
     \sum_{\ell\geq 0} (-1)^\ell\chi(\MH_{*,\ell}(G)) \ ,
     \end{equation}
     where $\chi(\MH_{*,\ell}(G))$ denotes the Euler characteristic of the $\ell$-component of magnitude homology $\MH_{*,\ell}(G)$. As magnitude homology is the first page in the magnitude-path spectral sequence, the quantity in Eq.~\eqref{eq:dec} is the Euler characteristic of the whole first page. The Euler characteristic is preserved by changing the pages of the spectral sequence, hence it agrees with the Euler characteristic of the limit object; that is,
     \[
     \sum_{\ell\geq 0} (-1)^\ell\chi(\MH_{*,\ell}(G))=\chi(\Reach_G) \ .
     \]
     As $G$ is an undirected connected graph, the category $\Reach_G$ is always contractible, hence its Euler characteristic is $1$. 
     
     On the other hand, the decategorification of the eulerian magnitude homology evaluated at $q=-1$ yields the quantity
    \begin{equation*}
     \sum_{\ell\geq 0} (-1)^\ell\chi(\EMH_{*,\ell}(G)) \ ,
     \end{equation*}
     which is the Euler characteristic of the first page in the regular magnitude-path spectral sequence. Its limit object is the complex of injective words on $|V(G)|>1$ letters, and we have 
       \[
     \sum_{\ell\geq 0} (-1)^\ell\chi(\EMH_{*,\ell}(G))=\chi(\Inj(V(G))) \ .
     \]
     By Theorem~\ref{thm:homotopycomplexinj}, $\chi(\Inj(V(G)))\neq 1$ for any undirected connected graph $G$ which is not a point, whereas $\chi(\Reach_H)=1$ for any connected graph $H$. It follows that there exists no graph $H$ such that $ \sum_{\ell\geq 0} \chi(\EMH_{*,\ell}(G))q^\ell=\# H(q)$.
\end{proof}

As to the authors' knowledge the regular magnitude of a graph had not appeared before, we conclude with the following question;

\begin{question}
    Which properties does $\#_r G(q)$ satisfy?
\end{question}

\bibliographystyle{alpha}
\bibliography{biblio}
 
\end{document}